\documentclass[11pt]{article}
\usepackage{amsmath,amssymb,amsthm}
\usepackage{hyperref} 
\usepackage{tikz-cd} 
\theoremstyle{definition} 
\newtheorem{definition}{Definition}[section]
\newtheorem{remark}[definition]{Remark}
\theoremstyle{plain}
\newtheorem{theorem}[definition]{Theorem}
\newtheorem{proposition}[definition]{Proposition}
\newtheorem{corollary}[definition]{Corollary}
\newtheorem{lemma}[definition]{Lemma}
\newtheorem{example}[definition]{Example} \title{Generalized Ueda Obstruction Classes and Applications to Non‑semi-positivity of Line Bundles} 
\author{Xiaojun Wu} 
\date{\today} 
\begin{document} 
\def\A{\mathcal{A}}
\def\cI{\mathcal{I}}
\def\Z{\mathbb{Z}}
\def\Q{\mathbb{Q}}  \def\C{\mathbb{C}}
 \def\R{\mathbb{R}}
 \def\N{\mathbb{N}}
 \def\H{\mathbb{H}}
  \def\P{\mathbb{P}}
 \def\rC{\mathrm{C}}
  \def\d{\partial}
 \def\dbar{{\overline{\partial}}}
\def\dzbar{{\overline{dz}}}
\def \ddbar {\partial \overline{\partial}}
\def\cB{\mathcal{B}}
\def\cD{\mathcal{D}}  \def\cO{\mathcal{O}}
\def\cbarO{\overline{\mathcal{O}}}
\def\D{\mathcal{D}}
\def\cC{\mathcal{C}}
\def\cF{\mathcal{F}}
\def \rank{\mathrm{rank}}
\def \deg{\mathrm{deg}}
\def \tot{\mathrm{Tot}}
\def \id{\mathrm{id}}
\bibliographystyle{plain}
\def \End{\mathrm{End}}
\def \dim{\mathrm{dim}}
\def \div{\mathrm{div}}
\def \ker{\mathrm{Ker}}
\def \im{\mathrm{Im}}
\def \rC{\mathrm{Cone}}
\newcommand{\Ub}{\mathcal{U}}
\newcommand{\dcech}{\check{\delta}}
\newcommand{\dcechg}{\delta\!\!\!\check{\delta}}
\newcommand{\lc}{\mathcal{L}}
\newcommand{\ec}{\mathcal{E}}
\maketitle
\begin{abstract}
We introduce a generalized first Ueda obstruction class for a line bundle along a closed analytic subvariety of a complex manifold, allowing singular subvarieties and restrictions that are not unitary flat. Using a Dolbeault resolution, the Chern curvature appears naturally, and non-vanishing of the class obstructs the existence of a smooth Hermitian metric with semi-positive curvature. As an application, we recover several classical examples within a unified framework and provide new examples of nef but not semi-positive line bundles.
We also introduce higher degree analogues of these obstruction classes.
\end{abstract}
 \smallskip \noindent
\\
\textbf{Key words and phrases.} Ueda theory, semi-positive line bundle, Projective bundle.
%Here we introduce a generalized first Ueda obstruction class associated with a holomorphic line bundle. In particular, we establish a variant of \cite[Theorem 1.4]{Koi21}, providing a criterion for non-semi-positivity. The advantage of this construction is that it can be applied in the setting of a closed analytic subvariety of a complex manifold. As an application, we discuss the case of the section at infinity in the compactification of a vector bundle. This completes the analysis of Serre's example, Grauert's example \cite[p. 365–366]{Gra62} and Kim's example \cite[Example 2.14]{Kim07}.
%Note that the non-semi-positivity 
\section{Introduction}
It is a fundamental fact in complex geometry that Kodaira’s vanishing and embedding theorem identifies strictly positive line bundles with ample line bundles: a holomorphic line bundle on a compact complex manifold admits a smooth Hermitian metric with strictly positive curvature if and only if it is ample in the sense of algebraic geometry. This equivalence underlies much of the classical interplay between differential geometry and algebraic geometry.

Motivated by this correspondence, Fujita asked in \cite{Fuj83} whether the weaker algebraic notion of nefness similarly implies the existence of a smooth Hermitian metric with semi-positive curvature.  This informal conjectural expectation was later shown to fail in general: Demailly, Peternell, and Schneider \cite[Example~1.7]{DPS94} constructed explicit examples of nef line bundles on compact complex manifolds that do not admit smooth metrics with semi-positive curvature. 

In \cite{Ued83}, T. Ueda studied neighborhoods of compact complex curves with topologically trivial normal bundle, introducing cohomological invariants, now called Ueda classes, which measure the obstruction to analytically linearizing such neighborhoods. These classes arise from the formal linearization of transition functions along the submanifold and detect when a formally flat structure cannot be integrated analytically. While Ueda’s original work focused on curves, subsequent developments by T. Koike \cite{Koi18,Koi21} extended Ueda theory to higher-codimensional smooth submanifolds with unitary flat normal bundles. In this generalized setting, the obstruction classes have been applied to derive criteria relevant to semi-positivity problems for Hermitian line bundles restricted to these submanifolds \cite[Theorem 1.4]{Koi21}.

Despite these advances, prior formulations of Ueda obstruction classes typically assume that the restriction of the line bundle \(L|_Y\) is unitary flat and that the subvariety \(Y\) is either smooth or has only special singularities, such as a finite number of nodes or cusps \cite{Ued91, Koi17, KU24}.  Moreover, proofs of non-semi-positivity in these contexts have often relied on case-by-case arguments, rather than a unified cohomological criterion.

In this paper we introduce a generalized first Ueda obstruction class associated with a holomorphic line bundle on a smooth complex manifold \(X\) along a reduced analytic subvariety \(Y\subset X\) (see Definition \ref{ueda_first}). In particular, we establish a variant of \cite[Theorem~1.4]{Koi21}, which provides a criterion for non-semi-positivity (see Corollary \ref{koi}). 
More precisely, our main result shows that non-vanishing of the generalized first Ueda obstruction class obstructs the existence of a smooth semi-positive metric.
The advantage of this construction is that it can be defined in the setting of a closed analytic subvariety of a complex manifold without requiring \(Y\) to be smooth or \(L|_Y\) to be unitary flat. 
Since the argument is essentially pointwise, the generalized Ueda obstruction class in fact obstructs semipositivity already at the level of continuous curvature forms; smoothness plays no essential role (see Remark~\ref{rem6}).

As an application, we study projectivizations of vector bundles arising from non-splitting extensions. 
In the case where the quotient bundle has rank one (equivalently, when the associated submanifold is a section), 
we show that a natural twist of the tautological line bundle on $\P(E)$ is not semi-positive (see Theorem~\ref{proj-bund}). 
This phenomenon is specific to the rank-one situation: in higher rank, the analogous statement on $\P(E)$ fails 
in general (see Remark~\ref{rem4}). 
This unifies the analysis of Serre’s example (see  Example \ref{ex4}), Grauert’s example \cite[p.~365--366]{Gra62} (see Example \ref{ex5}), and Kim’s example \cite[Example~2.14]{Kim07}. The non-semi-positivity of the line bundle in Grauert's example follows from a very recent result of Y. Zhang \cite{Zhang25}, while the non-semi-positivity of the anticanonical line bundle in Serre's example was established in \cite[Example 1.7]{DPS94}.
We can also recover some results of \cite[Theorem 1.2]{CH25} (see Example \ref{cao-horing2}).
Our basic observation is the following: given a non-splitting short exact sequence of vector bundles, the generalized first Ueda obstruction class of a suitable twist of the tautological line bundle on the projectivization of the middle term is, up to sign, represented by the second fundamental form of the sequence. In particular, this shows that the obstruction class is non-trivial.

To define the generalized first Ueda obstruction class, we use the logarithmic exact sequence of sheaves of abelian groups on \(X\) and carefully distinguish between the inverse image functor in the category of sheaves of abelian groups and the inverse image functor for \(\mathcal O_X\)-modules. This yields a cohomological invariant
\[
u_1(Y,X,L)\;\in\;H^1\!\bigl(Y,\,i^*\Omega_X^1\bigr),
\]
where \(i\colon Y\hookrightarrow X\) is the inclusion. This generalized class retains information from normal directions that would be lost if one worked only in the category of \(\cO_Y\)-modules. As cohomological invariants taking values in different sheaves, the generalized obstruction class and the classical first Ueda class do not coincide. There is, however, a natural map from the latter to the former (see Remark \ref{rem2}). Importantly, non-vanishing of the generalized class, which appears in the criterion for non-semi-positivity of \(L\), implies non-vanishing of the classical Ueda class.

The novelty of our approach compared to the classical Ueda-theoretic argument lies in the construction of the obstruction class. In the classical setting, the Ueda class is defined via transition functions on a suitable open cover, requiring \v{C}ech cohomology computations, and the role of the Chern class remains somewhat implicit. In contrast, our construction uses a resolution of the relevant sheaf via the Dolbeault complex, so that the Chern curvature of the line bundle naturally provides a globally defined  representative (see Remark \ref{rem1}). This perspective also explains why our generalized first Ueda obstruction class is insensitive to the singularities of the subvariety.

One conceptual advantage of the generalized obstruction class is that its nontriviality provides a simple cohomological criterion for the non-semi-positivity of $L$, thereby unifying and extending the various case-by-case arguments for projective bundles discussed above. In particular, for projectivizations of vector bundles, differential-geometric calculations show that the Chern curvature representative of the generalized obstruction class can be identified uniformly with the second fundamental form of the corresponding vector bundle extension. Its nontriviality therefore obstructs the existence of a semi-positive Hermitian metric by Corollary \ref{koi}.
Most previously known results are either restricted to surfaces or rely on sophisticated tools such as positivity of direct images.

Our construction also produces a systematic family of examples in the non-K\"ahler setting (see Corollary \ref{nonkah-ex}). More precisely, starting from a compact complex parallelisable manifold $R$ with $H^{0,1}(R)\neq 0$, one can consider non-trivial extensions
$$0 \to \Omega^1_R \to E \to \cO_R \to 0$$
and the associated projectivised bundle $X:=\mathbb{P}(E)$ endowed with the tautological line bundle $L:=\cO_{\mathbb{P}(E)}(1)$. In this situation, $L$ is nef, while it does not admit any smooth metric with semi-positive Chern curvature. This yields a general family of nef line bundles over non-K\"ahler manifolds exhibiting a strong discrepancy between numerical effectiveness and curvature positivity. To the author’s knowledge, such examples do not seem to have been explicitly constructed in the literature.

At the end of the paper, based on the breakthrough of \cite{AL19} on a Dolbeault--Grothendieck lemma over non-reduced analytic schemes, we define generalized higher degree Ueda obstruction classes (see Definition \ref{ueda}). 
This yields cohomological invariants for any \(k \geq 1\),
\[
u_k(Y,X,L)\;\in\;H^1\!\bigl(Y,\,i^{-1}\Omega_X^1\otimes_{i^{-1} \cO_X} \cO_X/\cI_{Y}^k \bigr),
\]
where \(i\colon Y\hookrightarrow X\) is the inclusion.
In particular, we obtain a variant of \cite[Theorem 3.7]{Koi21} (Corollary \ref{2nd-ueda}), which provides another criterion for the non-semi-positivity of a holomorphic line bundle in terms of generalized Ueda obstruction classes.
Compared with the construction in \cite[Section 3.2]{Koi21} (in particular \cite[Lemma 3.1]{Koi21}), our generalized Ueda obstruction classes are always well defined, independently of the choice of local coordinates or local trivializations of the line bundle.
Moreover, these generalized Ueda obstruction classes also admit representatives induced by the Chern curvature (see Remark \ref{curv-rep}).
Moreover, there exists a natural linear map sending the Ueda obstruction classes defined in \cite[Section 3.2]{Koi21} to our classes whenever the former are well defined (see Remark \ref{coincide} for the precise statement).

The paper is organized as follows. In Section~2 we construct the generalized first Ueda obstruction class and compare it with the classical case. Then we prove the variant of the non-semi-positivity criterion. In Section~3 we apply this criterion to projectivized vector bundles and complete the analysis of examples.
We divide this section into two subsections.
The first subsection is devoted to Theorem \ref{proj-bund}; we also present several applications, including new proofs of classical examples and new examples of nef but not semi-positive line bundles.
The second subsection discusses two possible generalizations to the higher-rank case: one fails, while the other does not yield interesting examples. We also present a variant of \cite[Theorem 1.2]{CH25}.
In Section 4, we introduce the definition of generalized higher degree Ueda obstruction classes and establish another variant of the non-semi-positivity criterion.
 \paragraph{}
\textbf{Acknowledgements.} I would like to thank my postdoctoral advisor, Professor Takayuki Koike, and Dr.~Yangyang Zhang for several helpful and stimulating discussions related to this work. This research was supported by the JSPS Postdoctoral Fellowships for Research in Japan (Standard). I am also grateful to Osaka Metropolitan University and the Osaka Central Advanced Mathematical Institute (OCAMI) for providing an excellent research environment.

\section{Construction of generalized first Ueda obstruction class}
We now turn to the construction of the generalized first Ueda obstruction class.
The general setting is as follows. Let \(X\) be a smooth complex manifold and let \(Y \subset X\) be a reduced (not necessarily smooth) analytic subvariety. Let \(L\) be a holomorphic line bundle on \(X\), and denote by \[ i \colon Y \hookrightarrow X \] the closed embedding. Our objective is to define a generalized first Ueda obstruction class \[ u_1(Y,X,L) \;\in\; H^1\!\left(Y,\, i^*\Omega_X^1\right). \] In contrast with the framework of \cite{Koi21}, we impose neither a smoothness assumption on \(Y\) nor a Hermitian flatness condition on \(L|_Y\).

Consider the exact sequence in the category of sheaves of abelian groups 
\[ 0 \;\longrightarrow\; \mathbb C^{*} \;\longrightarrow\; \mathcal O_X^{*} \;\xrightarrow{\;d\log\;} d\mathcal O_X \;\longrightarrow\; 0, \] 
where \(d\mathcal O_X \subset \Omega_X^{1}\) denotes the subsheaf of exact holomorphic \(1\)-forms, and the morphism \[ d\log \colon \mathcal O_X^{*} \longrightarrow d\mathcal O_X \] is given by the logarithmic differential \(f \mapsto \tfrac{df}{f}\).

Denote by \(i^{-1}\) the inverse image functor on the category of sheaves of abelian groups. Applying \(i^{-1}\) to the logarithmic exact sequence on \(X\), we obtain a short exact sequence of sheaves of abelian groups on \(Y\): \[ 0 \;\longrightarrow\;  i^{-1}\C^{*}_X \;\longrightarrow\; i^{-1}\mathcal O_X^{*} \;\xrightarrow{\;d\log\;} i^{-1} d\mathcal O_X \;\longrightarrow\; 0 . \]
Because \(Y\) is closed in \(X\), we have \(i^{-1}\mathbb{C}_X^* \simeq \mathbb{C}_Y^*\).

Let \(L\) be a holomorphic line bundle. By restricting its transition functions to \(Y\), the line bundle \(L\) determines a cohomology class \[ [L] \in H^{1}\!\left(Y,\; i^{-1}\mathcal O_X^{*}\right). \] A coordinate-free description is obtained as follows. In the category of sheaves of abelian groups on \(X\), there is a canonical morphism of sheaves \[ \mathcal O_X^{*} \longrightarrow i_* i^{-1}\mathcal O_X^{*}, \] induced by the adjunction between \(i^{-1}\) and \(i_*\). Passing to cohomology, this yields a natural map \[ H^{1}\!\left(X,\; \mathcal O_X^{*}\right) \longrightarrow H^{1}\!\left(X,\; i_* i^{-1}\mathcal O_X^{*}\right). \] Since \(Y\) is closed in \(X\), the functor \(i_*\) is exact in the category of sheaves of abelian groups. Therefore, the adjunction isomorphism induces a canonical identification \[ H^{1}\!\left(Y,\; i^{-1}\mathcal O_X^{*}\right) \;\simeq\; H^{1}\!\left(X,\; i_* i^{-1}\mathcal O_X^{*}\right). \] Under this identification, the class \([L]\) is precisely the image of the class of \(L\) in \(H^{1}(X,\mathcal O_X^{*})\) via the natural morphism above.

We emphasize that the reader should be careful about the choice of category. If one works in the category of $\cO_Y-$module sheaves, then tensoring with the structure sheaf of a submanifold $Y$ may, in a certain sense, result in the loss of information coming from the normal direction.
Here \(i^{-1}\) denotes the inverse image functor for sheaves of abelian groups, while \(i^*\) denotes the inverse image functor for \(\cO_X\)-modules (i.e. $i^* \cF:=i^{-1} \cF \otimes_{i^{-1} \cO_X} \cO_Y$).
The natural morphism $i^{-1}\mathcal O_X \to \mathcal O_Y$ induces a canonical morphism
$i^{-1}\cF \to i^*\cF$ for any $\cO_X$-module $\cF$.

\begin{definition}
\label{ueda_first}
We define the generalized first Ueda obstruction class \[ u_1(Y,X,L) \;\in\; H^{1}\!\left(Y,\;i^{*}\Omega_X^{1}\right) \] to be the image of \([L]\) under the composition \[ H^{1}\!\left(Y,\; i^{-1}\mathcal O_X^{*}\right) \;\xrightarrow{\;d\log\;} H^{1}\!\left(Y,\; i^{-1} d\mathcal O_X\right) \;\longrightarrow\; H^{1}\!\left(Y,\; i^{-1}\Omega_X^{1}\right) \;\xrightarrow{\;} H^{1}\!\left(Y,\; i^{*}\Omega_X^{1}\right), \] 
where: \(d\log\) is induced by the logarithmic differential \(f \mapsto \tfrac{df}{f}\);  the second arrow is the natural inclusion \(d\mathcal O_X \hookrightarrow \Omega_X^{1}\);  \(i^{*}\) denotes the inverse image functor in the category of \(\mathcal O_Y\)-modules. 
\end{definition}
For the convenience of the reader, we recall the construction of \(d\log\) in Definition \ref{ueda_first}.
Note that there is a canonical morphism of sheaves \[ \mathcal O_X^{*} \longrightarrow i_* i^{-1}\mathcal O_X^{*}, \] which arises from the natural transformation between the identity functor and the functor \(i_* i^{-1}\) in the category of sheaves of abelian groups. In particular, applying the logarithmic differential yields an induced morphism of sheaves \[ d\log \colon i_* i^{-1}\mathcal O_X^{*} \longrightarrow i_* i^{-1}\Omega_X^{1}. \] 

Now let us compare our definition with that of \cite{Koi21}.

As in the setting of \cite{Koi21},
assume that \(Y\) is smooth and that the restriction \(L|_Y\) is a unitary flat holomorphic line bundle.

 Consider the canonical exact sequence of locally free sheaves on \(Y\) \[ 0 \;\longrightarrow\; N^{-1}_{Y/X} \;\longrightarrow\; i^{*}\Omega_X^{1} \;\longrightarrow\; \Omega_Y^{1} \;\longrightarrow\; 0 . \] Since \(L|_Y\) is unitary flat, its first Chern class vanishes, and hence the image of \(u_1(Y,X,L)\) in \(H^{1}\!\left(Y,\Omega_Y^{1}\right)\) is zero. Indeed, this follows from the commutative diagram \[ \begin{tikzcd} H^{1}\!\left(Y,\, i^{-1}\mathcal O_X^{*}\right) \arrow[r] \arrow[d] & H^{1}\!\left(Y,\, \mathcal O_Y^{*}\right) \arrow[d, "c_1"] \\ H^{1}\!\left(Y,\, i^{*}\Omega_X^{1}\right) \arrow[r] & H^{1}\!\left(Y,\, \Omega_Y^{1}\right), \end{tikzcd} \] since \(c_1(L|_Y)=0\). 
 
 Consequently, the class \(u_1(Y,X,L)\) lies in the kernel of the natural map \[ H^{1}\!\left(Y, i^{*}\Omega_X^{1}\right) \;\longrightarrow\; H^{1}\!\left(Y,\Omega_Y^{1}\right), \] and therefore belongs to \[ \mathrm{Im}\!\left( H^{1}\!\left(Y, N^{-1}_{Y/X}\right) \;\longrightarrow\; H^{1}\!\left(Y, i^{*}\Omega_X^{1}\right) \right). \] In general, the natural map \(H^{1}\!\left(Y, N^{-1}_{Y/X}\right) \to H^{1}\!\left(Y, i^{*}\Omega_X^{1}\right)\) is not injective, so a lift in $H^{1}\!\left(Y, N^{-1}_{Y/X}\right)$ is not unique. However, by the explicit definition of the Ueda class as a \v{C}ech cocycle given by the logarithmic differentials of the transition functions (see, for example, \cite[Section 2.2]{Koi21}), the same formula defines a well-defined element of \(H^{1}\!\left(Y, N^{-1}_{Y/X}\right)\) which thus provides a lift. For this reason, we refer to the class \[ u_1(Y,X,L) \;\in\; \mathrm{Im}\!\left( H^{1}\!\left(Y, N^{-1}_{Y/X}\right) \;\longrightarrow\; H^{1}\!\left(Y, i^{*}\Omega_X^{1}\right) \right) \] as the generalized first Ueda obstruction class. This definition remains useful in view of the following result \cite[Theorem 1.4]{Koi21}. \begin{theorem}
 \label{koike_thm} Let \(X\) be a complex manifold, \(Y \subset X\) a compact K\"ahler submanifold, and \(L\) a line bundle on \(X\) such that \(L|_Y\) is topologically trivial. If the first Ueda obstruction class is nontrivial, then \(L\) is not semi-positive. \end{theorem} 

Here we present a variant of Theorem~\ref{koike_thm}, which can follow directly from that result. Since our argument naturally involves the Chern curvature and is global in nature, we include the details for completeness.
We will see later that the same argument extends to singular subvarieties and to higher-degree obstruction classes. 

For later use, we recall several results concerning the (derived) category of sheaves of abelian groups.
A comprehensive reference for this material is \cite{KS02}.
\begin{definition}
\label{soft_sheaf}
Let \(X\) be a topological space. A sheaf \(\mathcal F\) of abelian groups on \(X\) is called soft if for every closed subset \(Z \subset X\) with inclusion $i_Z: Z \to X$, the restriction map \[ \Gamma(X,\mathcal F)\;\longrightarrow\;\Gamma(Z,i_Z^{-1}\mathcal F) \] is surjective; that is, every section of \(\mathcal F\) defined on a closed subset \(Z\) extends to a global section on \(X\).
\end{definition}
A standard verification shows that every fine sheaf is soft. Moreover, if \(W \subset X\) is a closed subset and \(\mathcal F\) is a soft sheaf on \(X\), then its inverse image \(i^{-1}_W \mathcal F\) is a soft sheaf on \(W\).

Recall the following standard lemma.
\begin{lemma}(cf. e.g.\cite[Exercise II.5]{KS02})
\label{soft-lem}
Let \(X\) be a paracompact space and let \(\mathcal F\) be a soft sheaf on \(X\). Then \[ H^i(X,\mathcal F)=0 \qquad \text{for all } i>0 . \] In particular, if \(\mathcal F^\bullet\) is a bounded complex of soft sheaves on \(X\), its hypercohomology is computed by the complex of global sections \(\Gamma(X,\mathcal F^\bullet)\).
\end{lemma}
We now reinterpret the generalized Ueda obstruction class in terms of curvature via the Dolbeault resolution and give the following variant of \cite[Theorem 1.4]{Koi21}.
 \begin{corollary}
\label{koi}
 Let \(X\) be a complex manifold, \(Y \subset X\) a compact submanifold, and \(L|_Y\) a unitary flat line bundle on \(Y\). If the generalized first Ueda obstruction class is nontrivial, then \(L\) is not semi-positive. \end{corollary}
 \begin{proof}
We prove by contradiction.
Assume that there exists a smooth metric $(L,h)$ such that the Chern curvature $\Theta(L,h)$ is semi-positive. 
 
Consider the commutative diagram \[ \begin{tikzcd} H^{1}\!\left(X,\, \mathcal O_X^{*}\right) \arrow[r] \arrow[d, "c_1"] & H^{1}\!\left(Y,\, i^{-1}\mathcal O_X^{*}\right) \arrow[d] \\ H^{1}\!\left(X,\, \Omega_X^{1}\right) \arrow[r] & H^{1}\!\left(Y,\, i^{-1}\Omega_X^{1}\right), \end{tikzcd} \]
By the Chern-Weil theorem,  $\Theta(L,h)$
represents the first Chern class of $L$ via the fine resolution \(\mathcal E_X^{1,\bullet}\) defined as follows.
Consider the bounded complex \(\mathcal E^{p,\bullet}_X\) of sheaves of smooth differential forms on \(X\), which provides a fine resolution of the holomorphic \(p\)-form sheaf \(\Omega^p_X\).
(We stress that the smoothness assumption on \(X\) is essential for the argument. This is the reason why we assume \(X\) to be smooth throughout the paper.)

Consider the bounded complex \(i^{-1}\mathcal E^{p,\bullet}_X\) obtained from sheaves of smooth differential forms on \(X\), which furnishes a soft resolution of the sheaf \(i^{-1}\Omega^p_X\). By Lemma~\ref{soft-lem},  the cohomology group $H^1(Y, i^{-1}\Omega^p_X)$ can be computed by the complex of global sections of \(i^{-1}\mathcal E^{p,\bullet}_X\).

Note that \(i^{-1}\mathcal E^{p,\bullet}_X\) is not a module over \(\mathcal E^{0,0}_Y\); consequently, it does not constitute a fine resolution.

The collection of stalks \(\{\Theta(L,h)(z)\}_{z\in Y}\) glues to a global section in \( \Gamma\bigl(Y, i^{-1}\mathcal E^{1,1}_X\bigr), \) 
which
via the above soft resolution, represents the image of \(c_1(L)\in H^{1}(X,\Omega^1_X)\) in \( H^{1}\bigl(Y, i^{-1}\Omega^1_X\bigr). \) This section should not be confused with \(i^*\Theta(L,h)\), which represents a class in \(H^{1}(Y,\Omega^1_Y)\).

Since \(i^{-1}\Omega^1_X\) is a locally free \(i^{-1}\mathcal O_X\)-module, the bounded complex \[ i^{-1}\Omega^1_X \otimes_{i^{-1}\mathcal O_X} \mathcal E^{0,\bullet}_Y \] is a fine resolution of \[ i^{-1}\Omega^1_X \otimes_{i^{-1}\mathcal O_X} \mathcal O_Y \;=\; i^*\Omega^1_X, \] because \(\mathcal E^{0,\bullet}_Y\) is a fine resolution of \(\mathcal O_Y\).

Note that there is a natural identification \[ i^{-1}\mathcal E^{1,\bullet}_X \;\simeq\; i^{-1}\Omega^1_X \otimes_{i^{-1}\mathcal O_X} i^{-1}\mathcal E^{0,\bullet}_X . \] Via this identification, the canonical morphism \[ i^{-1}\mathcal E^{0,\bullet}_X \longrightarrow \mathcal E^{0,\bullet}_Y \] induces a morphism of complexes \[ i^{-1}\mathcal E^{1,\bullet}_X \longrightarrow i^{*}\mathcal E^{1,\bullet}_X . \]
Via this resolution, this morphism of complexes induces the map in hypercohomology such that the image of \(\Theta(L,h)\) represents the generalized first Ueda obstruction class in \( H^{1}\bigl(Y, i^{*}\Omega^1_X\bigr). \)

Note that, up to this point, we have only used the assumptions that \(Y\) is closed, \(X\) is smooth, and \(L\) is a holomorphic line bundle over \(X\).

By contradiction, suppose that \(\Theta(L,h)\) is semi-positive. Then its restriction \(i^{*}\Theta(L,h)\) is also semi-positive. On the other hand, \(i^{*}\Theta(L,h)\) represents \(c_{1}(L|_{Y}) = 0 \in H^{1}(Y,\Omega^{1}_{Y})\). Let \(\omega\) be a Gauduchon metric on \(Y\) (whose existence is established in \cite{Gau77}). The semi-positive \((\dim Y,\dim Y)\)-form \[ i^{*}\Theta(L,h)\wedge \omega^{\dim Y-1} \] represents the zero class in \(H^{\dim Y,\dim Y}(Y)\). It follows that this form vanishes identically, and hence \(i^{*}\Theta(L,h)\equiv 0\).

At a smooth point \(z \in Y\), choose local holomorphic coordinates \((h,v)\) on \(X\) such that \(Y\) is locally defined by \(\{v=0\}\). In these coordinates, the Chern curvature \(\Theta(L,h)(z)\) is represented by a block matrix of the form \[ \begin{pmatrix} 0 & A \\ \overline{A}^{\,t} & * \end{pmatrix}. \] 
Since \(i^*\Theta(L,h)=0\), the tangential–tangential block vanishes.
Since \(\Theta(L,h)(z)\) is semi-positive, it follows that \(A=0\). Consequently, \[ \{\Theta(L,h)(z)\}_{z\in Y}=0, \] which contradicts the non-triviality of the generalized first Ueda obstruction class.
 \end{proof}
The above Corollary~\ref{koi} can be easily generalized to singular analytic subvarieties, as shown in the next remark.
\begin{remark}
\label{rem1}
One may weaken the assumption that \(Y\) is a compact submanifold to the condition that \(Y\) is a compact, closed analytic subvariety. In the proof, the smoothness of \(Y\) is used only to apply the fact that the complex $\mathcal{E}_Y^{0,\bullet}$ is a fine resolution of $\cO_Y$ and ensure the existence of a Gauduchon metric. 

When \(Y\) is possibly singular, the Dolbeault–Grothendieck lemma is highly nontrivial. 
 Since \(Y\) is reduced, there exists a complex of fine sheaves  \(\mathcal A_Y^{0,\bullet}\),  consisting of currents and containing the complex of smooth forms  \(\mathcal E_Y^{0,\bullet}\),  which provides a resolution of \(\mathcal O_Y\).  Such a complex was introduced in \cite[Section 12]{AS12}. 
When $Y$ is smooth, this complex coincides with the complex of sheaves of smooth forms.  
The pullback of smooth forms induces a morphism of complexes  \[ i^{-1}\mathcal E_X^{0,\bullet} \longrightarrow \mathcal A_Y^{0,\bullet}.\]
(More precisely, we use the soft resolution \(i^{-1}\Omega^1_X \otimes_{i^{-1} \cO_X} \mathcal A_Y^{0,\bullet}\) of $i^* \Omega^1_X$ to compute the generalized first Ueda obstruction class.)

Let \(\pi \colon \widetilde Y \to Y\) be a resolution of singularities. Each connected component of \(\widetilde Y\) admits a Gauduchon metric. Applying the above argument on \(\widetilde Y\), we conclude that \(\pi^{*} i^{*}\Theta(L,h) \equiv 0\). Since \(i^{*}\Theta(L,h)\) is the restriction of a smooth form on \(X\), the vanishing of its pullback implies that \(i^{*}\Theta(L,h)\) itself vanishes identically on \(Y\). Hence the same contradiction argument applies without any smoothness assumption on \(Y\).
\end{remark}
The same arguments as in Remark \ref{rem1} yield the following statement without assuming compactness.
\begin{remark}
\label{rem3}
Let \(X\) be a complex manifold and \(Y \subset X\) a closed analytic subvariety.
Let $L$ be a line bundle over $X$.
Assume that there exists a smooth metric $h$ on $X$ whose restriction $h|_{Y}$ is unitary flat over the regular part of $Y$.
Then the generalized first Ueda obstruction class of $L$ is trivial.
\end{remark}
\begin{remark}
\label{rem6}
By applying the H\"older estimates for the $\dbar$-equation on sufficiently small Euclidean balls (see, for example, \cite[Theorem 2.2.2]{HL84}), it follows that the complex of sheaves consisting of $(0,\bullet)$-forms $u$ whose coefficients, as well as those of $\dbar u$, are of class $C^k$ ($k\in\N$) is a fine resolution of the sheaf of holomorphic functions.
Fix $r \geq 2$, so that the Chern curvature of a $C^r$-metric has continuous coefficients.
In Corollary \ref{koi}, replacing the sheaves of smooth forms by the corresponding sheaves of forms with $C^{r-2}$ coefficients, the same proof shows that there is no $C^r$-metric with semi-positive Chern curvature if the generalized first Ueda obstruction class is non-zero.

We emphasize that the essential regularity requirement in the above argument is that the Chern curvature admits continuous coefficients. Indeed, continuous differential forms on the ambient space $X$ admit a natural restriction to the smooth submanifold $Y$, and the pointwise semipositivity argument used in the proof of Corollary~\ref{koi} remains valid in this setting.
\end{remark}
We briefly compare this argument with Koike’s proof of Theorem~\cite[Theorem~1.4]{Koi21}.
In particular, we discuss the difficulties that arise in extending Koike's proof to singular subvarieties.
\begin{remark}
\label{rem2}
In Koike’s setting, the submanifold \(Y\) is assumed to be smooth and the restriction \(L|_Y\) is unitary flat; the proof proceeds via an explicit analysis of local transition functions in a neighborhood of \(Y\). In contrast, our approach relies on the choice of soft resolutions of \(i^{-1}\Omega_X^1\) by sheaves of smooth differential forms, through which the Chern curvature appears naturally and globally. From this perspective, Koike’s argument can be understood as a specific choice of a \v{C}ech complex for computing the relevant cohomology groups. 
When the subvariety is singular, the choice of covering becomes more subtle, and in general there seems to be no natural choice, especially when the singular locus is non-isolated, as we discuss in Remark \ref{rem5}.

With our definition of the generalized first Ueda obstruction class, the method applies to reduced analytic subvarieties \(Y\) without any smoothness assumption. In particular, when \(Y\) is singular, the conormal sheaf is no longer a subsheaf of \(i^{*}\Omega_X^1\), and it is therefore unclear how to define Ueda obstruction classes with values in the conormal sheaf. 
(More precisely,
the exterior derivative \( d\colon \mathcal O_X \longrightarrow \Omega_X^1 \) maps \(\mathcal I_Y\) into \(\Omega_X^1\), and since \(d(\mathcal I_Y^2)\subset \mathcal I_Y\cdot \Omega_X^1\), it induces a well-defined morphism \[\mathcal{N}^*_{Y/X}:= \mathcal I_Y / \mathcal I_Y^2 \;\longrightarrow\; \Omega_X^1 \otimes_{\mathcal O_X} \mathcal O_Y \;=\; i^*\Omega_X^1 \]
which is not necessarily injective.) 
This provides further motivation for our formulation in terms of \(i^{*}\Omega_X^1\).
\end{remark} 
To the best of the authors’ knowledge, the following summarizes the existing references in the literature on Ueda obstruction classes over singular analytic varieties.
\begin{remark}
\label{rem5}
In \cite{Ued91}, Ueda developed his theory for curves on surfaces under the assumption that the singularities are mild, such as nodal singularities. The specific type of singularity is essential in his approach. Similar theories were developed in \cite{Koi23, KU24}, where analogous assumptions on the nature of the singularities are likewise indispensable.

In the literature, the subvariety under consideration is a curve whose singular points are isolated. One may therefore choose local coordinate charts so that each singular point is contained in a single chart. Consequently, the \v{C}ech cocycle representing the Ueda obstruction classes can be chosen to have support disjoint from the singular points.
\end{remark}
At the end of this section, using an example due to Neeman, we show that Corollary~\ref{koi} is not a necessary and sufficient condition.
Recall first the following special case \cite[Corollary 9.2]{Nee}.
\begin{proposition}
\label{Nee}
Consider the following commutative diagram of holomorphic maps of connected complex manifolds:
\[
\begin{array}{ccc}
M & \xrightarrow{i_{\tilde{X}}} & \tilde{X} \\
\;\;\;\;\;\downarrow \mathrm{id} & & \downarrow \pi \\
M & \xrightarrow{i_X} & X
\end{array}
\]
where $i_{\tilde{X}}$ and $i_X$ are closed immersions, $M$ is compact and K\"ahler, and $\dim X = \dim \tilde{X} = \dim M + 1$. Assume that the normal bundle of $M$ in $X$ is trivial. Suppose further that the map $\pi : \tilde{X} \to X$ is a cyclic covering of degree $r$ with $r \geq 2$. If $u_1(M, X, \cO_X(M)) \neq 0$, then
\[
u_1(M, \tilde{X}, \cO_{\tilde{X}}(M)) = 0.
\]
\end{proposition}

\begin{proof}
This follows from \cite[Corollary 9.2]{Nee} together with the observation that the natural map
\[
H^1(M, N_{M/X}^{-1}) \to H^1(M, i_X^* \Omega^1_X)
\]
(resp. $H^1(M, N_{M/\tilde{X}}^{-1}) \to H^1(M, i_{\tilde{X}}^* \Omega^1_{\tilde{X}})$)
sends the Ueda class to the corresponding generalized first Ueda obstruction class $u_1(M, X, \cO_X(M))$ (resp. $u_1(M, \tilde{X}, \cO_{\tilde{X}}(M))$).
\end{proof}
\begin{remark}
Consider the same setting as in Proposition \ref{Nee}. Since the covering is cyclic and $\cO_{\tilde{X}}(M) = \pi^* \cO_X(M)$, the line bundle $\cO_{\tilde{X}}(M)$ is semi-positive if and only if $\cO_X(M)$ is semi-positive. However, by Proposition \ref{Nee}, their generalized first Ueda obstruction classes do not vanish simultaneously if we take $(M,X)$ to be Serre's example.
\end{remark}
\section{Projectivization of vector bundle}
In this section, we apply Corollary~\ref{koi} together with the Chern curvature representative of the generalized first Ueda obstruction class to establish the non-semi-positivity of the tautological line bundle on certain projectivizations of vector bundles.

Let \(R\) be a compact complex manifold. Let \(Q\) be a holomorphic vector bundle and \(S\) a holomorphic vector bundle on \(R\). Assume that \[ H^{1}\!\left(R,\operatorname{Hom}(Q,S)\right)\neq 0, \] and consider a nontrivial extension of holomorphic vector bundles \begin{equation}
\label{exact2}
 0 \longrightarrow S \longrightarrow E \longrightarrow Q \longrightarrow 0 . 
 \end{equation} 
 Set \[ X := \mathbb P(E), \qquad Y := \mathbb P(Q), \] where \(Y\) is a smooth complex submanifold of \(X\). Let \[ L := \mathcal O_{\mathbb P(E)}(1) \] denote the tautological line bundle on \(X\). 
\subsection{The case $\mathrm{rank}\;Q=1$} 
%Assume from now on that 
%$S=\cO_R$ and 
%$S$ and 
In this subsection, we always assume that $Q$ is a line bundle.
Note that $Y$ is isomorphic to $R$.

Let \( \pi:\P(E)\to R,\pi_Q:\P(Q)\to R \) be the natural projections.  
On \(\P(Q)\) we have the universal quotient \( \pi_Q^*Q \twoheadrightarrow  L|_{\P(Q)} . \)
We claim that \(L \otimes \pi^* Q^*\) is not semi-positive.
Note that the restriction of \(L \otimes \pi^* Q^*\) on $\P(Q)$ is trivial.

Let us first calculate the tangent bundle of a projective bundle. By definition of \(\P(E)\), the relative tangent sheaf fits into a canonical identification \[ T_{\P(E)/X} \;\cong\; \mathcal H om\!\bigl(\mathcal K_E,\;L \bigr), \] where \(\mathcal K_E:=\ker(\pi^*E\to L)\).
Restricting to \(\P(Q)\), \[ i^*T_{\P(E)/X} \;\cong\; \mathcal H om\!\bigl(i^*\mathcal K_E,\; i^*L\bigr). \] On \(\P(Q)\), the composite \( \pi_Q^*E \longrightarrow \pi_Q^*Q \longrightarrow L|_{\P(Q)} \) is a quotient of \(\pi_Q^*E\). 
Note that \( i^*\mathcal K_E \;=\; \ker(\pi_Q^*E\to L|_{\P(Q)}). \) From the exact sequence \(0\to S\to E\to Q\to0\), we get an exact sequence on \(\P(Q)\): 
\begin{equation}
\label{exact1} 
0 \longrightarrow \pi_Q^* S \longrightarrow i^*\mathcal K_E \longrightarrow \ker\!\bigl(\pi_Q^*Q \to L|_{\mathbb P(Q)}\bigr) \longrightarrow 0 .
\end{equation}
By the same intrinsic description, \[ T_{\P(Q)/X} \;\cong\; \mathcal H om\!\bigl(\ker(\pi_Q^*Q\to  L|_{\P(Q)}),\;L|_{\P(Q)}\bigr). \]
By definition, \[ N_{\P(Q)/\P(E)} \;=\; \frac{i^*T_{\P(E)}}{T_{\P(Q)}}\simeq \frac{i^*T_{\P(E)/X}}{T_{\P(Q)/X}}. \] 
Substitute the Hom descriptions: \[ N_{\P(Q)/\P(E)} \;\cong\; \frac{ \mathcal H om(i^*\mathcal K_E, L|_{\P(Q)}) }{ \mathcal H om(\ker(\pi_Q^*Q\to L|_{\P(Q)}), L|_{\P(Q)}) }. \] 
Apply \(\mathcal H om(-,L|_{\P(Q)})\) to the short exact sequence (\ref{exact1}). 
Since all sheaves are locally free, \(\mathcal H om(-,L|_{\P(Q)})\) is exact, yielding \[ 0 \longrightarrow \mathcal H om(\ker(\pi_Q^*Q\to L|_{\P(Q)}),L|_{\P(Q)}) \longrightarrow \mathcal H om(i^*\mathcal K_E,L|_{\P(Q)}) \longrightarrow \mathcal H om(\pi_Q^* S,L|_{\P(Q)}) \longrightarrow 0. \] Therefore, \[  N_{\P(Q)/\P(E)} \;\cong\; \mathcal H om(\pi_Q^* S,L|_{\P(Q)})\simeq\; \pi_Q^* S^\vee \otimes \mathcal O_{\P(Q)}(1). \]

Since \(Q\) is a line bundle, the submanifold \(Y\) is isomorphic to \(R\). 
Via this identification, the conormal bundle \(N^{*}_{\mathbb P(Q)/\mathbb P(E)}\) corresponds to \(\mathcal H om(Q,S)\). 
In what follows, to simplify notation, we identify differential forms on \(Y\) with those on \(R\).
\begin{remark}
Consider the exact sequence
\[
0 \to \pi^* \Omega_R^1 \to \Omega_X^1 \to \Omega_{X/R}^1 \to 0 .
\]

Applying $\pi_*$ gives
\[
0 \to \pi_*\pi^*\Omega_R^1 \to \pi_*\Omega_X^1 \to \pi_*\Omega_{X/R}^1 .
\]

Since the fibers of $\pi$ are projective spaces which have no non-trivial holomorphic forms, we have
$
\pi_*\Omega_{X/R}^1 = 0 .
$
On the other hand, by the projection formula and the fact that $\pi_*\mathcal{O}_X=\mathcal{O}_R$, we obtain
$
\pi_*\pi^*\Omega_R^1
=
\Omega_R^1 .
$
Thus
\[
\pi_*\Omega_X^1 \simeq \Omega_R^1 .
\]

Taking global sections yields
\[
H^0(X,\Omega_X^1)
\simeq
H^0(R,\pi_*\Omega_X^1)
\simeq
H^0(R,\Omega_R^1),
\] which implies that the restriction morphism \[ H^{0}\!\left(Y, i^{*}\Omega^{1}_{X}\right) \;\longrightarrow\; H^{0}\!\left(Y,\Omega^{1}_{Y}\right) \] is surjective. Consequently, the natural morphism \[ H^{1}\!\left(Y, N^{*}_{\mathbb P(Q)/\mathbb P(E)}\right) \;\longrightarrow\; H^{1}\!\left(Y, i^{*}\Omega^{1}_{X}\right) \] is injective. In particular, the nontriviality of the generalized first Ueda obstruction class is equivalent to the nontriviality of the (classical) first Ueda obstruction class.
\end{remark}

Via the above identification $Y \simeq R$, the exact sequence~\eqref{exact2} corresponds to \[ 0 \longrightarrow \mathcal K_E|_Y \longrightarrow \pi^*E|_Y \longrightarrow L|_Y \longrightarrow 0  \]
(since they both correspond to the mapping cone of the bundle morphism $E \to Q$).
 
Fix a smooth Hermitian metric on $E$ which induces a smooth metric on \(\pi^*E\). Then \(\mathcal K_E\) and \(L\) (resp. $S$ and $Q$) are endowed with the induced and quotient metrics, respectively. In particular, as a smooth vector bundle, \(\pi^* E\) is the orthogonal direct sum of \(\mathcal K_E\) and \(L\)  (resp. \(E\) is the orthogonal direct sum of \(S\) and \(Q\)). We consider the second fundamental form, in the sense of \cite[Section~14, Chapter~V]{Dem12}, \[ \beta^* \in \mathcal E_X^{0,1}\bigl(\mathcal H om(L,\mathcal K_E)\bigr). \]

Its restriction \(\beta^*|_Y\) represents the extension class in $ H^1\!\left(Y,\mathcal H om(Q,S)\right) $ associated with the exact sequence~\eqref{exact2}. 

In particular, the cohomology class of \(\beta^*\) in $ H^1\!\left(X,\mathcal H om(L,\mathcal K_E)\right) $ is nontrivial as its restriction to the submanifold $Y$ isomorphic to $R$.
We also consider the second fundamental form \[ \beta_R^* \in \mathcal E_R^{0,1}\bigl(\mathcal H om(Q,S)\bigr). \] 
We emphasize that these second fundamental forms are globally defined.

We now compute the Chern curvature of \(L\) in a neighborhood of \(Y\). Let \(e_{S,1},\cdots,e_{S,r},  e_Q\) be a local frame of \(E\), where \(e_{S,1},\cdots,e_{S,r}\) is holomorphic and valued in \(S\), and \(e_Q\) is smooth and valued in \(Q\). 
We may assume that \(e_{S,1},\cdots,e_{S,r}\) are orthogonal to $e_Q$.
Let \(e^*_{S,1},\cdots,e^*_{S,r},  e_Q^*\) be the induced dual frame of \(E^*\), where \(e^*_{S,1},\cdots,e^*_{S,r}, ^*\) is smooth and valued in \(S^*\), and \(e_Q^*\) is holomorphic and valued in \(Q^*\).

Locally, over a Stein open set of $R$, the fibres of $\pi$ in \(X\) can be described as the space of hyperplanes defined by \[ v_1 e^*_{S,1}+\cdots+v_r e^*_{S,r} + v_0 e_Q^*, \qquad [v_1 :\cdots: v_r : v_0] \in \mathbb P^r. \] 
Consider the Stein chart \(\{v_0 \neq 0\}\), which lies in the complement of the smooth non-holomorphic hypersurface \(\mathbb P(S)\). In this chart, the fibres of $\pi$ in \(X\) are parametrized by hyperplanes of the form \[ w_1 e^*_{S,1}+\cdots+w_r e^*_{S,r} +  e_Q^*, \qquad w_i := v_i/v_0 \in \mathbb C,\qquad 1 \leq i \leq r. \] The submanifold \(Y\) is given by \(\{w_1=\cdots=w_r=0\}\).
In this chart, the section \[ w_1 e^*_{S,1}+\cdots+w_r e^*_{S,r} + e_Q^* \] defines a local smooth frame of \(L\). Consequently, the kernel \(\mathcal K_E=\ker(\pi^*E\to L)\) admits the local smooth frame \[ e_{S,1}-w_1 e_Q,\; \ldots,\; e_{S,r}-w_r e_Q . \] Upon restriction to \(Y\) (that is, setting \(w_1=\cdots=w_r=0\)), these sections become holomorphic and form a local holomorphic frame of \[ \mathcal K_E|_Y \;\simeq\; \pi_Q^* S \;\simeq\; S . \]

Taking the Chern connection \(D_{\pi^*E}\), we compute for each \(i\) \[ D_{\pi^*E}(e_{S,i}-w_i e_Q) = \pi^* D_E(e_{S,i}-w_i e_Q)\;-\; dw_i\, e_Q . \] 
In the right-hand side, we view \(e_{S,i}-w_i e_Q\) as a local smooth section of \(E\), pulled back to \(\P(E)\) via \(\pi\).
Restricting to a point \(z\in Y\) (that is, \(w_1=\cdots=w_r=0\)), the second fundamental form yields \[ \beta_{\pi^*E}(e_{S,i}) = \beta_R(e_{S,i}) - dw_i(z) \otimes \, e_Q . \]
This element lies in \( i^{-1}\Omega^1_X\big|_z \otimes Q\big|_z . \)

Taking adjoints, we obtain 
\[ \beta^*_{\pi^*E}(e_{Q}) = \beta^*_R(e_{Q}) - \sum_i d\bar{w}_i(z) \otimes \, e_{S,i} . \]
This element lies in \( i^{-1}\overline{\Omega}^1_X\big|_z \otimes S\big|_z . \)

%By taking adjoint, we have
%$$i^{-1} \beta^*=\beta^*_R-d\overline{v}.$$
By the curvature formula \cite[Theorem~14.5, Chapter~V]{Dem12}, with induced metric $h_Q$ on $Q$,
$$i^{-1} \Theta(L,h)=\Theta(Q,h_Q)+ \sum_i \big(dw_i \wedge d \bar{w}_i-dw_i \wedge \beta^*_{i,Q} -\beta_{Q,i} \wedge d \bar{w}_i\big) 
$$
with
$$\beta^*_{R}(e_{Q})=\sum_i \beta^*_{i,Q} e_{S,i},$$
$$\beta_{R}(e_{S,i})= \beta_{Q,i} e_{Q}.$$
%In particualr,
%the generalized Ueda first obstruction class
%is represented by
%$-\beta^* \in H^1(Y, N^*_{Y/X})$
%which is non trivial.

Thus, the inverse image (in the category of sheaves of abelian groups) of the Chern curvature of \(L \otimes \pi^{*}Q^{*}\) along \(Y\) is given pointwise by \[ \sum_i \Bigl( dw_i \wedge d\bar w_i - dw_i \wedge \beta^{*}_{Q,i} - \beta_{Q,i} \wedge d\bar w_i \Bigr). \] 
%where \(\beta_{Q,i}\) denotes the component of the second fundamental form with values in \(Q\). 
In particular, the generalized first Ueda obstruction class is represented, pointwise on \(Y\), by \(-\beta^{*}\). Hence the class \[ \{-\beta^{*}\} \;\in\; H^{1}\!\left(Y,\, N^{*}_{Y/X}\right) \] is nontrivial.

(The sign convention is consistent with that in \cite[Chapter V, Theorem 14.3 and Proposition 14.9]{Dem12}.)

For the reader's convenience, we summarize the calculations of this subsection in the following theorem.

\begin{theorem}
\label{proj-bund}
Let \(R\) be a compact complex manifold. Let \(Q\) be a holomorphic line bundle and \(S\) a holomorphic vector bundle on \(R\). Assume that
\[
H^{1}\!\left(R,\operatorname{Hom}(Q,S)\right)\neq 0,
\]
and consider a nontrivial extension of holomorphic vector bundles
\begin{equation}
0 \longrightarrow S \longrightarrow E \longrightarrow Q \longrightarrow 0.
\end{equation}
Set
\[
X := \mathbb{P}(E), \qquad Y := \mathbb{P}(Q),
\]
so that \(Y\) is a smooth complex submanifold of \(X\). Let
\[
L := \mathcal{O}_{\mathbb{P}(E)}(1)
\]
be the tautological line bundle on \(X\), and let \(\pi : \mathbb{P}(E) \to R\) denote the natural projection.

Then
\[
u_1\bigl(Y, X, L \otimes \pi^* Q^*\bigr) \neq 0,
\]
and hence \(L \otimes \pi^* Q^*\) is not semi-positive.
\end{theorem}
Using Theorem \ref{proj-bund}, one can construct many new examples of nef holomorphic line bundles that do not admit any smooth metric with semi-positive Chern curvature.
\begin{example}
\label{ex3}
Let $R$ be a compact complex manifold with nef cotangent bundle such that $H^{1,1}(R) \neq 0$. In particular, there exists a non-trivial extension of holomorphic vector bundles
$$0 \to \Omega^1_R \to E \to \cO_R \to 0.$$
Consider $X := \mathbb{P}(E)$ and $L := \cO_{\mathbb{P}(E)}(1)$. As an extension of nef vector bundles, $E$ is nef, which by definition is equivalent to the nefness of $L$. However, by Theorem \ref{proj-bund}, $L$ does not admit any smooth metric with semi-positive Chern curvature.
\end{example}
\begin{example}
\label{ex4}
Take $R$ to be an elliptic curve and let $E$ be the unique non-trivial extension of trivial line bundles in Example \ref{ex3}.
We recover the fact that the line bundle $\mathcal{O}_{\mathbb{P}(E)}(1)$ is nef but not semi-positive.
\end{example}
\begin{corollary}
Let $R$ be either a compact quotient of a bounded symmetric domain or a smooth complete intersection in a complex torus. Since $R$ is compact K\"ahler, one has $H^{1,1}(R) \neq 0$. 

In the first case, the cotangent bundle is ample. In the second case, the cotangent bundle is a quotient of the restriction of the cotangent bundle of the torus, hence is nef. 

Therefore, Example \ref{ex3} applies.
\end{corollary}
\begin{corollary}
\label{nonkah-ex}
Let $R$ be a compact complex parallelisable manifold, that is, a compact complex manifold with trivial holomorphic tangent bundle. According to the author's knowledge, this class was first systematically studied in \cite{Wang54}. 

Assume furthermore that $H^{0,1}(R) \neq 0$, which in particular implies that $H^{1,1}(R) \neq 0$. (For example, if $R$ is a solvmanifold, this condition holds by \cite[Remark 4.2]{Nak75}.) 

Therefore, Example \ref{ex3} applies. To the author's knowledge, this provides a systematic family of nef line bundles
over non-K\"ahler manifolds that do not admit any smooth metric with semi-positive
Chern curvature, and such constructions do not seem to appear explicitly in the literature.
\end{corollary}
We can also provide another proof of Grauert’s example.
\begin{example}
\label{ex5}
Let $R$ be a compact Riemann surface of genus $g \ge 2$. 
Let $F$ be a holomorphic line bundle on $R$ with $\deg F = 1$ such that 
$H^1(R,\mathcal{O}_R(F)) \neq 0$.
Consider a non-trivial extension of line bundles
\[
0 \to F \to E \to \mathcal{O}_R \to 0.
\]
Let $X$ be the compact complex surface $\mathbb{P}(E)$ and denote by $p : X \to R$ the natural projection.
Define a line bundle on $X$ by
$
L := \mathcal{O}_{\mathbb{P}(E)}(1).
$
As an extension of nef line bundles, $E$ is a nef vector bundle. Consequently, $L$ is nef.
On the other hand, its self-intersection satisfies
\[
\int_X c_1(L)^2
=
\int_R p_*\bigl(c_1(L)^2\bigr)
=
\int_R c_1(E)
>0,
\]
since $F$ is ample.
Therefore $L$ is nef and big, but not semi-positive by Theorem \ref{proj-bund}.
\end{example}
\subsection{The case $\mathrm{rank}\;Q \geq 2$}
In this subsection, we discuss the case $\operatorname{rank} Q \ge 2$.
In contrast to the rank-one case, the above approach does not extend directly when $\operatorname{rank} Q \ge 2$ (see Remark \ref{rem4}).
Using derived-category methods to relate the non-splitting of the original short exact sequence to that of a certain associated short exact sequence, one obtains a natural generalization of Theorem \ref{proj-bund}. However, this generalization does not yield interesting examples (see Remark \ref{proj-bund2}).

At the end of this section, we consider some special classes of $Q$ and derive several interesting examples.

Let us first discuss a possible generalization of Theorem \ref{proj-bund} to the case $\operatorname{rank} Q \ge 2$, presented in the following remark.
\begin{remark}
\label{rem4}
A possible generalization of Theorem \ref{proj-bund} in the case $\mathrm{rank}\,Q \ge 2$ would assert the non-semi-positivity of the line bundle
$\mathcal{O}_{\mathbb{P}(E)}(1) \otimes \pi_E^* \det(Q)^*$,
where $\pi_E : \mathbb{P}(E) \to R$ denotes the natural projection, under the same assumptions as in Theorem \ref{proj-bund2}. We show that such a generalization fails.

Assume that $R$ is projective. Consider the twist of the short exact sequence (\ref{exact2}) by a line bundle $L$, to be chosen later. Since $L$ has rank $1$, the twisted short exact sequence is still non-splitting.

Let $r := \mathrm{rank}\,Q$. Applying the above hypothetical generalization, one would obtain that the line bundle
$\mathcal{O}_{\mathbb{P}(E \otimes L)}(1) \otimes \pi_{E \otimes L}^* \det(Q \otimes L)^*$
is not semi-positive, where $\pi_{E \otimes L} : \mathbb{P}(E \otimes L) \to R$ denotes the natural projection.

Under the identification $\mathbb{P}(E \otimes L) \simeq \mathbb{P}(E)$, this line bundle corresponds to
$\mathcal{O}_{\mathbb{P}(E)}(1) \otimes \pi_E^* \det(Q)^* \otimes \pi_E^* L^{*(r-1)}$.
Thus it would follow that this line bundle is not semi-positive.

However, if one takes $L^*$ sufficiently ample, then since $\mathcal{O}_{\mathbb{P}(E)}(1)$ is $\pi_E$-ample, the line bundle
$\mathcal{O}_{\mathbb{P}(E)}(1) \otimes \pi_E^* \det(Q)^* \otimes \pi_E^* L^{*(r-1)}$
is ample, hence semi-positive. This is a contradiction.

Therefore, such a generalization does not hold.
\end{remark}

We now relate the original short exact sequence (\ref{exact2}) to the naturally associated short exact sequence (\ref{exact4}).

Let \(R\) be a complex manifold, and let \(Q\) and \(S\) be holomorphic vector bundles on \(R\).
For the derived-category arguments below, the compactness of \(R\) is not required. Thus, we work in a more general setting than at the beginning of this section.

Consider the cohomology group
$
H^{1}\!\left(R,\operatorname{Hom}(Q,S)\right),
$
whose elements correspond to extensions of holomorphic vector bundles of the form
\[
0 \longrightarrow S \longrightarrow E \longrightarrow Q \longrightarrow 0.
\]
Let $\pi : \mathbb{P}(E) \to R$ denote the natural projection, and set $L := \mathcal{O}_{\mathbb{P}(E)}(1)$, $Y:=\mathbb{P}(Q) \subset X:= \mathbb{P}(E)$.
 
Define
\(\mathcal K_E:=\ker(\pi^*E\to L)\) with $\pi: X\to R$ the natural projection.
Consider the exact sequence of vector bundles
\begin{equation}
\label{exact4} 0 \longrightarrow \mathcal K_E|_Y \longrightarrow (\pi^*E)|_Y \longrightarrow L|_Y \longrightarrow 0 . 
 \end{equation}

%In fact,
First, we study the behaviour of this short exact sequence under taking the derived direct image $R(\pi|_Y)_*$ in the (bounded) derived category of $\cO_Y$-modules.
In particular, we will show that
\[
(\pi|_Y)_*\mathcal K_E|_Y \simeq S.
\]
%$$
% 0 \longrightarrow (\pi|_Y)_{*} \mathcal K_E|_Y \longrightarrow (\pi|_Y)_{*} (\pi^*E)|_Y=E \longrightarrow (\pi|_Y)_{*} L|_Y=Q \longrightarrow 0 .
%$$
%Let us now give a detailed proof.
By relative Bott vanishing theorem (cf. \cite[Theorem 10.6, 10.7, Chap. VII]{Dem12}),  we have for any $j >0$,
$$R^j (\pi|_Y)_{*} L|_Y=0,$$
i.e. $R (\pi|_Y)_{*} L|_Y=Q$.
By projection formula and relative Bott vanishing theorem,  we have for any $j >0$,
$$R^j (\pi|_Y)_{*}  (\pi^*E)|_Y=0,$$
i.e. $R (\pi|_Y)_{*} (\pi^*E)|_Y=E$.
Consider the Euler sequence 
\[ 0 \to \Omega_{Y/R}^1(1) \to \pi|_Y^*Q \to \mathcal O_Y(1) \to 0. \]
There is a commutative diagram on \(Y=\mathbb P(Q)\): \begin{equation}\label{exact6}
\begin{tikzcd}
0 \ar[r]
& \mathcal K_E|_Y \ar[r] \ar[d]
& (\pi^*E)|_Y \ar[r] \ar[d]
& \mathcal O_Y(1) \ar[r] \ar[d,equal]
& 0 \\
0 \ar[r]
& \Omega_{Y/R}^1(1) \ar[r]
& (\pi^*Q)|_Y \ar[r]
& \mathcal O_Y(1) \ar[r]
& 0
\end{tikzcd}
\end{equation}
which implies the short exact sequence
\begin{equation}
\label{exact7}
0 \to (\pi|_{Y})^* S \to \mathcal K_E|_Y \to \Omega_{Y/R}^1(1) \to 0.
\end{equation}
Note that the map $\mathcal K_E|_Y \to \Omega_{Y/R}^1(1)$ is surjective, which yields the last short exact sequence (\ref{exact7}) by the snake lemma.
One may also deduce the short exact sequence (\ref{exact7}) from the natural map
\[
\begin{aligned}
H^1(R,\operatorname{Hom}(Q,S))
&\to H^1\bigl(Y,\operatorname{Hom}((\pi|_{Y})^*Q,(\pi|_{Y})^*S)\bigr) \\
&\to H^1\bigl(Y,\operatorname{Hom}(\Omega_{Y/R}^1(1),(\pi|_{Y})^*S)\bigr),
\end{aligned}
\]
induced by the morphism
$
\Omega_{Y/R}^1(1)\to (\pi|_{Y})^*Q.
$

By relative Bott vanishing theorem 
we have similarly that
$$R (\pi|_Y)_{*}\Omega_{Y/R}^1(1)=0$$
which implies that
$$R (\pi|_Y)_{*} \mathcal K_E|_Y \simeq R (\pi|_Y)_{*} (\pi|_{Y})^* S =S.$$
Applying $R(\pi|_{Y})_*$ to (\ref{exact4}) yields a distinguished triangle in the (bounded) derived category of $\cO_R$-modules
\[
R(\pi|_{Y})_*\mathcal K_E|_Y \longrightarrow R(\pi|_{Y})_*(\pi^*E)|_Y \longrightarrow R(\pi|_{Y})_* L \longrightarrow R(\pi|_{Y})_*\mathcal K_E|_Y[1],
\]
which corresponds to the short exact sequence (\ref{exact2}) by the preceding computations.

The short exact sequence (\ref{exact4}) determines an extension class in the bounded derived category of $\cO_Y$-module sheaves, namely a morphism in the in the bounded derived category
\[
L|_Y \longrightarrow \mathcal K_E|_Y[1].
\]
(An explicit description in terms of the second fundamental form can be given as follows.
Fix a metric on $E$ and the induced metric on $(\pi|_Y)^{*} E$.
Consider the fine resolution of $\mathcal K_E|_Y$ given by $(\mathcal K_E|_Y \otimes_{\cO_Y} \mathcal{E}^{(0, \bullet)}_Y, \dbar)$.
The morphism corresponding to the short exact sequence (\ref{exact4}) is represented by the second fundamental form, together with the identification
\[
\mathcal{E}^{(0,1)}_Y(\mathcal{H}om(L|_Y,\mathcal K_E|_Y))
\simeq
\mathcal{H}om(L|_Y, \mathcal{E}^{(0,1)}_Y(\mathcal K_E|_Y)).
\]
By a direct calculation, as in the previous subsection, one can show that the second fundamental form induced by the Hermitian metric on $E$ is non-trivial. However, we do not know a simple argument that proves the corresponding extension class is non-trivial without passing to the derived category framework.)

Consider the natural map
\[
\mathrm{Hom}(L|_Y,\mathcal K_E|_Y[1])
\longrightarrow
\mathrm{Hom}(R(\pi|_Y)_*L|_Y,\, R(\pi|_Y)_*\mathcal K_E|_Y[1])
\simeq
\mathrm{Hom}(Q,S[1]).
\]
where the right-hand side is regarded as an object of the bounded derived category of \(\cO_R\)-module sheaves,
under which the short exact sequence (\ref{exact4}) is sent to the short exact sequence (\ref{exact2}).
In particular, if the short exact sequence (\ref{exact2}) does not split, the class defined by (\ref{exact4}) is nonzero in $\mathrm{Hom}(L|_Y,\mathcal K_E|_Y[1])$,
which implies that (\ref{exact4}) does not split.
Note that when $Q$ is a line bundle, the short exact sequence (\ref{exact4}) coincides with the short exact sequence (\ref{exact2}) under the identification $Y \simeq R$.

We also consider the short exact sequence
\begin{equation}
\label{exact5}
0 \longrightarrow \mathcal K_E \longrightarrow \pi^*E \longrightarrow L \longrightarrow 0,
\end{equation}
which does not split, since its restriction to $Y$ is the short exact sequence (\ref{exact4}).
\begin{remark}
\label{proj-bund2}
By applying Theorem \ref{proj-bund} to the short exact sequences (\ref{exact4}) and (\ref{exact5}), we obtain the following statement.
Assume that $R$ is compact and the short sequence \ref{exact2} does not split.
Let $\pi_E : \mathbb{P}(E) \to R$ and $\pi_Q : \mathbb{P}(Q) \to R$ be the natural projections.
Consider the natural projections $\pi_1 : \mathbb{P}(\pi_E^* E) \to \mathbb{P}(E)$ and $\pi_2 : \mathbb{P}(\pi_Q^* Q) \to \mathbb{P}(Q)$. Then the line bundles
\[
\cO_{\mathbb{P}(\pi_E^* E)}(1) \otimes \pi_1^* \cO_{\mathbb{P}(E)}(-1), \qquad
\cO_{\mathbb{P}(\pi_Q^* Q)}(1) \otimes \pi_2^* \cO_{\mathbb{P}(Q)}(-1)
\]
are not semi-positive.
However, this conclusion can also be established easily by restricting the line bundles to the inverse images under $\pi_1$ and $\pi_2$ of a fiber of $\pi_E$ and $\pi_Q$, respectively. Therefore, this approach does not yield interesting examples of nef but non-semi-positive line bundles.
\end{remark}

However, one may obtain a non-semi-positivity result under the additional assumption that the vector bundle $Q$ admits a special quotient line bundle.
It follows directly from Theorem \ref{proj-bund}.
\begin{corollary}
\label{cao-horing1}
Let \(R\) be a compact complex manifold, and let \(Q\) and \(S\) be holomorphic vector bundles on \(R\). Assume that
\[
H^{1}\!\left(R,\operatorname{Hom}(Q,S)\right)\neq 0,
\]
and let
\[
0 \longrightarrow S \longrightarrow E \longrightarrow Q \longrightarrow 0
\]
be a nontrivial short exact sequence of holomorphic vector bundles.

Let \(Q'\) be a quotient line bundle of \(Q\). This induces a natural commutative diagram of short exact sequences
\[
\begin{tikzcd}
0 \arrow[r] & S \arrow[r] \arrow[d] & E \arrow[r] \arrow[d, "\mathrm{id}"] & Q \arrow[r] \arrow[d, two heads] & 0 \\
0 \arrow[r] & S' \arrow[r]          & E \arrow[r]                          & Q' \arrow[r]                     & 0
\end{tikzcd}
\]
and we assume that the induced lower short exact sequence is non-split.

Then the line bundle \(\mathcal{O}_{\mathbb{P}(E)}(1)\otimes \pi_E^* Q'^{*}\) is not semi-positive, where \(\pi_E : \mathbb{P}(E)\to R\) denotes the natural projection.
\end{corollary}
As an application, we obtain the following variant of \cite[Theorem 3.1]{CH25}.
\begin{corollary}
\label{cao-horing2}
Let \(R\) be a compact complex manifold such that \(H^1(R,\mathcal{O}_R)\neq 0\). Let
\[
0 \longrightarrow \mathcal{O}_R^{r_1} \longrightarrow E \longrightarrow \mathcal{O}_R^{r_2} \longrightarrow 0
\]
be a nontrivial short exact sequence of holomorphic vector bundles. Then the tautological line bundle \(\mathcal{O}_{\mathbb{P}(E)}(1)\) is not semi-positive.
\end{corollary}
\begin{proof}
By assumption,
\[
H^1\!\left(R,\operatorname{Hom}(\mathcal{O}_R^{r_2}, \mathcal{O}_R^{r_1})\right)
\simeq H^1(R,\mathcal{O}_R)^{\oplus r_1 \times r_2},
\]
so the given class can be represented by a nonzero \(r_1 \times r_2\) matrix with entries in \(H^1(R,\mathcal{O}_R)\).
By fixing a basis of \(H^1(R,\mathcal{O}_R)\), one can identify the class with a nonzero
\(  (r_1 \cdot h^1(R,\mathcal{O}_R))\times r_2\) complex-valued matrix.
In particular, there exists a row which cannot be written as a complex linear combination of the other rows.
Equivalently, under the original identification as an \(H^1(R,\mathcal{O}_R)\)-valued matrix, there exists a row which cannot be written as a complex linear combination of the others.
 
Let \(Q'\) be the direct summand of \(\mathcal{O}_R^{r_2}\) corresponding to such a row. Write
\[
\mathcal{O}_R^{r_2} = Q' \oplus Q'^{\perp}.
\]
Under the same notation as in Corollary \ref{cao-horing1}, we obtain a short exact sequence
\[
0 \to \mathcal{O}_R^{r_1} \to S' \to Q'^{\perp} \to 0,
\]
where \(S'\) denotes the kernel of the natural surjection \(E \to Q'\).

Fix a smooth metric on \(E\). Endow all the above vector bundles with the metrics induced by this fixed metric.
Let \(\{e_{\mathcal{O}_R^{r_1},i}\}_{1 \le i \le r_1}\) and \(\{e_{\mathcal{O}_R^{r_2},j}\}_{1 \le j \le r_2}\) be the standard bases of \(\mathcal{O}_R^{r_1}\) and \(\mathcal{O}_R^{r_2}\), respectively.
Without loss of generality, we may assume that \(Q'\) corresponds to \(j=1\).
Let \[-\beta^* \in \ec^{0,1}(R,\operatorname{Hom}(\mathcal{O}_R^{r_2},\mathcal{O}_R^{r_1}))\] be the second fundamental form associated with the above exact sequence with respect to the fixed metric.

Under the above choice of basis, the extension class of
$$0 \to S' \to E \to Q' \to 0$$
in
\(
H^1\!\left(R,\operatorname{Hom}(Q',S')\right)
\)
is represented by a \(\mathcal{E}^{0,1}(R)\)-valued column vector of length \(r_1 + r_2 - 1\), whose entries are given by the coefficients of \(-\beta^*(e_{\mathcal{O}_R^{r_2},1})\) with respect to the basis elements \(e_{\mathcal{O}_R^{r_1},i}\) and \(e_{\mathcal{O}_R^{r_2},j}\) for \(j \neq 1\). In fact, only the components along \(e_{\mathcal{O}_R^{r_1},i}\) are nonzero, since \(\beta^*(e_{\mathcal{O}_R^{r_2},1})\) takes values in \(\mathcal{O}_R^{r_1}\).

Thus the extension class is the image of the class of \(-\beta^*(e_{\mathcal{O}_R^{r_2},1})\) in
\[
H^1\!\left(R,\operatorname{Hom}(Q',\mathcal{O}_R^{r_1})\right)
\]
under the natural morphism
\[
H^1\!\left(R,\operatorname{Hom}(Q',\mathcal{O}_R^{r_1})\right)
\longrightarrow
H^1\!\left(R,\operatorname{Hom}(Q',S')\right).
\]
Note that there is an exact sequence
\[
H^0(R,S') \longrightarrow H^0(R,Q'^{\perp}) \longrightarrow H^1(R,\mathcal{O}_R^{r_1})
= H^1\!\left(R,\operatorname{Hom}(Q',\mathcal{O}_R^{r_1})\right).
\]
To apply Corollary \ref{cao-horing1}, it suffices to show that the class of \(-\beta^*(e_{\mathcal{O}_R^{r_2},1})\) does not lie in the kernel of
\[
H^1\!\left(R,\operatorname{Hom}(Q',\mathcal{O}_R^{r_1})\right)
\longrightarrow
H^1\!\left(R,\operatorname{Hom}(Q',S')\right).
\]
Equivalently, this class does not lie in the image of the connecting morphism
\[
H^0(R,Q'^{\perp}) \longrightarrow H^1(R,\mathcal{O}_R^{r_1})
= H^1\!\left(R,\operatorname{Hom}(Q',\mathcal{O}_R^{r_1})\right).
\]
On the other hand, since \(Q'^{\perp}\) is trivial and globally generated by the sections \(e_{\mathcal{O}_R^{r_2},j}\) for \(j \neq 1\), the image of the connecting morphism is generated by the classes \(-\beta^*(e_{\mathcal{O}_R^{r_2},j})\) for \(j \neq 1\).
Note that by the Dolbeault description of the extension class via the second fundamental form, the connecting morphism sends $e_{\mathcal{O}_R^{r_2},j}$ to the class of $-\beta^*(e_{\mathcal{O}_R^{r_2},j})$ for $j \neq 1$.
 By our choice of \(Q'\), the class of \(-\beta^*(e_{\mathcal{O}_R^{r_2},1})\) does not lie in the image of
\[
H^0(R,Q'^{\perp}) \longrightarrow H^1(R,\mathcal{O}_R^{r_1})
= H^1\!\left(R,\operatorname{Hom}(Q',\mathcal{O}_R^{r_1})\right),
\]
which concludes the proof.
\end{proof}
\begin{remark}
Let us compare Corollary \ref{cao-horing2} with \cite[Theorem 3.1]{CH25}.

First, the assumption that \(R\) is compact K\"ahler is replaced by the assumption that \(R\) is an arbitrary compact complex manifold. In \cite{CH25}, the K\"ahler condition is used in an essential way to apply deep results of \cite{Sim,Deng}, which imply that \(E\) is a local system.
Moreover, our arguments show that there is no smooth metric on 
$\mathcal{O}_{\mathbb{P}(E)}(1)$ which is semi-positive in a 
neighborhood of some subvariety (for example, $\mathbb{P}(\mathcal{O}_R^{r_2})$) 
without being globally semi-positive.

Second, when \(h^0(R,E)=1\), \cite[Theorem 3.1]{CH25} gives a complete description of singular semi-positive metrics on \(\mathcal{O}_{\mathbb{P}(E)}(1)\), whereas our result only implies that \(\mathcal{O}_{\mathbb{P}(E)}(1)\) is not semi-positive.

Third, when \(h^0(R,E)\geq 2\) and \(r_1=1\), \cite[Theorem 3.1]{CH25} gives a detailed description of singular semi-positive metrics on \(\mathcal{O}_{\mathbb{P}(E)}(1)\). Our result implies, in particular, that in this situation these metrics cannot be smooth.

Fourth, as shown in \cite[Example 3.10]{CH25}, their statement does not hold when \(r_1 \geq 2\). Corollary \ref{cao-horing2} provides partial information in this case.

Fifth, our argument avoids the use of direct image techniques \cite{CP17,PT} and the numerical flatness criterion \cite{Wu22}.
It seems difficult to use the positivity of the direct image to 
deduce the lack of semi-positivity near a closed subvariety, 
as described above.
\end{remark}
\section{Higher degree case}
We now return to the general framework and extend the construction of the obstruction classes to higher order.

To define generalized higher-degree Ueda obstruction classes, we use a Dolbeault--Grothendieck lemma due to \cite{AL19}, which provides a suitable generalization to sheaves of smooth forms over analytic schemes that are not necessarily reduced.
We begin by introducing the generalized higher-degree Ueda obstruction classes.
Then, using the Dolbeault--Grothendieck lemma, we give explicit representatives of these classes induced by the Chern curvature.
Finally, we provide another criterion for non-semi-positivity.

To begin with, we recall some basic notions in formal geometry that will be used to define higher-degree Ueda obstruction classes.

Let $Y$ be a reduced analytic subvariety of a complex manifold $X$, and let $\cI_Y$ be its ideal sheaf. For $i \geq 1$, consider the $i$-th infinitesimal neighborhood of $Y$ in $X$, defined by
\[
(Y^{(i)}, \cO_{Y^{(i)}}) := (Y, \cO_X / \cI_Y^i).
\]
For integers $j \le i$, the inclusion $\cI_Y^i \subset \cI_Y^j$ induces a natural surjective morphism of ringed spaces
\[
\pi_{i,j} : (Y^{(i)}, \cO_{Y^{(i)}}) \longrightarrow (Y^{(j)}, \cO_{Y^{(j)}}),
\]
whose structure sheaf morphism is given by the canonical projection
\[
\cO_X / \cI_Y^i \longrightarrow \cO_X / \cI_Y^j.
\]
For each $i \ge 1$, there is a natural closed immersion
\[
\iota_i : (Y^{(i)}, \cO_{Y^{(i)}}) \longrightarrow (X, \cO_X),
\]
induced by the quotient morphism of sheaves $\cO_X \to \cO_X / \cI_Y^i$.

For integers $j \le i$, the morphisms $\pi_{i,j}$ and $\iota_i$ are compatible in the sense that
\[
\iota_j \circ \pi_{i,j} = \iota_i,
\]
i.e. the diagram
\[
(Y^{(i)}, \cO_{Y^{(i)}}) \xrightarrow{\ \pi_{i,j}\ } (Y^{(j)}, \cO_{Y^{(j)}})
\]
\[
\qquad\ \ \searrow \iota_i \quad\quad\quad\ \ \downarrow \iota_j
\]
\[
\qquad\qquad\qquad (X, \cO_X)
\]
is commutative.
Note that the inclusion $i : Y \to X$ coincides with $\iota_1$.
Moreover, we have that
\[
\iota_i^{*}(-)=\cO_{Y^{(i)}} \otimes_{i^{-1}\cO_X} i^{-1}(-).
\]
Since we compute sheaf cohomology on the same underlying topological space, we will occasionally make no distinction between $Y^{(i)}$ and $Y$.
However, we emphasize that $\iota_i^{*}$ denotes the inverse image functor in the category of $\cO_{Y^{(i)}}$-modules, which depends on $i$, and is distinct from the inverse image functor $\iota_i^{-1}=\iota_1^{-1}$ in the category of sheaves of abelian groups.
\begin{definition}
\label{ueda}
We define the generalized $i-$th Ueda obstruction class \[ u_i(Y,X,L) \;\in\; H^{1}\!\left(Y^{(i)},\;\iota_i^{*}\Omega_X^{1}\right) \] to be the image of \([L]\) under the composition \[ H^{1}\!\left(Y,\; i^{-1}\mathcal O_X^{*}\right) \;\xrightarrow{\;d\log\;} H^{1}\!\left(Y,\; i^{-1} d\mathcal O_X\right) \;\longrightarrow\; H^{1}\!\left(Y,\; i^{-1}\Omega_X^{1}\right) \;\xrightarrow{\;} H^{1}\!\left(Y^{(i)},\; \iota_i^{*}\Omega_X^{1}\right), \] 
where: \(d\log\) is induced by the logarithmic differential \(f \mapsto \tfrac{df}{f}\);  the second arrow is the natural inclusion \(d\mathcal O_X \hookrightarrow \Omega_X^{1}\);  \(\iota_i^{*}\) denotes the inverse image functor in the category of \(\mathcal O_{Y^{(i)}}\)-modules. 
\end{definition}
When \(i=1\), this definition agrees with Definition \ref{ueda_first}.

Recall the following fundamental Dolbeault--Grothendieck lemma for non-reduced analytic subvarieties (\cite[Theorem 1.1, Lemma 4.8, Definition 11.2]{AL19}).

\begin{theorem}
\label{AL}
Assume that $Y$ is of pure dimension $n$. Then for any $i \geq 1$, there exist fine sheaves $\ec^{(0,\bullet)}_{Y^{(i)}}$ such that the following sequence is exact:
\[
0 \to \cO_{Y^{(i)}} \to \ec^{(0,0)}_{Y^{(i)}} \xrightarrow{\dbar} \cdots \xrightarrow{\dbar} \ec^{(0,n)}_{Y^{(i)}} \to 0
\]
with canonical morphisms
$$i^{-1}\ec^{(0,\bullet)}_{X} \to \ec^{(0,\bullet)}_{Y^{(i)}}.$$
Over $Y^{(i)}_{\mathrm{reg}}$ (i.e.\ the locus where $Y$ is smooth and where $\cO_{Y^{(i)}}$ is Cohen--Macaulay), for any $k \geq 0$, one has
\[
\ec^{(0,k)}_{Y^{(i)}} = \ec^{(0,k)}_{X} \big/ \big( \ec^{(0,k)}_X \cI_Y^i + \ec^{(0,k)}_X \overline{\cI}_Y + \ec^{(0,k-1)}_X \, d\overline{\cI}_Y \big)
\]
and the canonical morphisms are induced by the quotient maps.
\end{theorem}
\begin{remark}
In \cite{AL19}, the authors denote the above sheaves $\ec^{(0,\bullet)}_{Y^{(i)}}$ by $\mathcal{A}^{(0,\bullet)}_{Y^{(i)}}$. The sheaves $\ec^{(0,\bullet)}_{Y^{(i)}}$ consist of smooth forms on a non-reduced subvariety, whereas $\mathcal{A}^{(0,\bullet)}_{Y^{(i)}}$ are suitable subsheaves of currents containing $\ec^{(0,\bullet)}_{Y^{(i)}}$. 
It is well known that the Dolbeault--Grothendieck lemma fails for the complex of sheaves of smooth differential forms on a singular variety.
For our purposes, the precise definition of $\mathcal{A}^{(0,\bullet)}_{Y^{(i)}}$ is not needed, since we will only consider the case where $Y$ is a smooth subvariety. Therefore, to simplify notation, we do not adopt their convention.
\end{remark}
For the convenience of the reader, we recall the local description of the corresponding sheaves near a smooth point.
\begin{example}
\label{ex1}
(\cite[Section 4.1]{AL19})
Let $X=\C^{d+1}_{(h^1, \cdots, h^d, v)}$ and let $Y=\{v=0\}$ be a smooth hypersurface. Then, for any $i\ge 1$, one has
\[
\ec^{(0,k)}_{Y^{(i)}} \simeq \bigoplus_{0 \le j \le i-1} v^j \ec^{(0,k)}_Y.
\]
The canonical morphisms
\[
i^{-1}\ec^{(0,\bullet)}_{X} \to \ec^{(0,\bullet)}_{Y^{(i)}}
\]
are induced by taking the coefficients $\sum_{j=0}^{i-1} \varphi_j(h^1, \cdots, h^d)v^j$ in the Taylor expansion in the $v$-direction (with remainder, see also \cite[(2.13)]{AL19}) of the coefficients of a smooth form $\varphi$, namely
\[
\varphi=\sum_{j=0}^{i-1} \varphi_j(h^1, \cdots, h^d)v^j + O(|v|^i) + O(\overline{v}, d\overline{v}).
\]
\end{example}
\begin{remark}
\label{non-reduced-diff}
When $Y$ is reduced, one has
\[
\ec^{(0,\bullet)}_{Y} \simeq i^{-1}(\ec^{(0,\bullet)}_{X} \big/ \mathrm{Ker}\big( \ec^{(0,\bullet)}_{X} \to i_{Y_{\mathrm{reg}},*}\ec^{(0,\bullet)}_{Y_{\mathrm{reg}}}\big)),
\]
where $i_{Y_{\mathrm{reg}}}$ denotes the inclusion of $Y_{\mathrm{reg}}$ into $X$.
In particular, a smooth form on $Y$ is uniquely determined by its restriction to the regular locus $Y_{\mathrm{reg}}$.
When $Y$ is non-reduced, it is unclear whether a smooth form on $Y$ is uniquely determined by its restriction to the regular locus of $Y$.
\end{remark}
We now describe representatives of the generalized higher-degree Ueda obstruction classes given by the Chern curvature.
\begin{remark}
\label{curv-rep}
Assume that $Y$ is a closed analytic subvariety of pure dimension in a complex manifold $X$.
Fix $j \ge 1$. Since $i^{-1}\Omega^1_X$ is a locally free $i^{-1}\cO_X$-module, Theorem \ref{AL} implies that the bounded complex
\[
i^{-1}\Omega^1_X \otimes_{i^{-1}\cO_X} \ec^{0,\bullet}_{Y^{(j)}}
\]
is a fine resolution of
\[
i^{-1}\Omega^1_X \otimes_{i^{-1}\cO_X} \cO_{Y^{(j)}} = \iota_j^*\Omega^1_X,
\]
since $\ec^{0,\bullet}_{Y^{(j)}}$ is a fine resolution of $\cO_{Y^{(j)}}$.

Let $h$ be a smooth metric on a holomorphic line bundle $L$ on $X$. As in Remark \ref{rem1}, we obtain a Chern curvature representative of generalized higher-degree Ueda obstruction classes as follows. The collection of stalks $\{\Theta(L,h)(z)\}_{z \in Y}$ glues to a global section in
$
\Gamma\bigl(Y, i^{-1}\ec^{1,1}_X\bigr),
$
which represents the image of $c_1(L) \in H^1(X,\Omega^1_X)$ in $H^1\bigl(Y, i^{-1}\Omega^1_X\bigr)$, as in the proof of Corollary \ref{koi}.

There is a natural identification
\[
i^{-1}\ec^{1,\bullet}_X \simeq i^{-1}\Omega^1_X \otimes_{i^{-1}\cO_X} i^{-1}\ec^{0,\bullet}_X.
\]
Via this identification, the canonical morphism constructed in Theorem \ref{AL}
\[
i^{-1}\ec^{0,\bullet}_X \longrightarrow \ec^{0,\bullet}_{Y^{(j)}}
\]
induces a morphism of complexes
\[
i^{-1}\ec^{1,\bullet}_X \longrightarrow i^{-1}\Omega^1_X \otimes_{i^{-1}\cO_X} \ec^{0,\bullet}_{Y^{(j)}},
\]
which is by Theorem \ref{AL} quasi-isomorphic to the natural morphism
\[
i^{-1}\Omega^1_X \longrightarrow \iota_j^*\Omega^1_X.
\]
Via this resolution, the image of $\Theta(L,h)$ represents the higher-degree Ueda obstruction classes in
$
H^1\bigl(Y^{(j)}, \iota_j^*\Omega^1_X\bigr).
$
\end{remark}
We now compare Definition \ref{ueda} with Koike's definition given in \cite[Section 3.2]{Koi21}.
\begin{remark}
\label{coincide}
Our definition \ref{ueda}
is compatible with Koike’s definition \cite[Section 3.2]{Koi21} for a holomorphic line bundle \(L\) on a complex manifold \(X\) containing a smooth hypersurface \(Y\) such that \(L|_Y\) is unitary flat, via the natural morphisms. 
More precisely, consider
for any \(n \geq 2\),
\[
H^1(Y, N^{-n}_{Y/X}) \to H^1(Y, i^* \Omega^1_X \otimes N^{-(n-1)}_{Y/X}) \to H^1\big(Y, i^{-1} \Omega^1_X \otimes_{i^{-1} \mathcal{O}_X} \mathcal{O}_X / \mathcal{I}_Y^n \big),
\]
where \(\mathcal{I}_Y\) denotes the ideal sheaf of \(Y\).

The first morphism is induced by the inclusion \(N^{-1}_{Y/X} \hookrightarrow i^*\Omega^1_X\), and the second by the canonical identification and the natural inclusion
\[
N^{-(n-1)}_{Y/X})\simeq (\cI_Y/\cI_Y^2)^{\otimes (n-1)} \xrightarrow{\ \sim\ } \cI_Y^{n-1}/\cI_Y^n \hookrightarrow \mathcal{O}_X / \mathcal{I}_Y^n,
\]
where the canonical identification is given by multiplication in the ideal.

We now explain the compatibility. Adopting the notation and definition of the obstruction classes of \cite[Section 3.2]{Koi21}, take a sufficiently fine finite open covering \(\{U_j\}\) of \(Y\) and sufficiently small Stein open subsets \(V_j \subset X\) such that \(V_j \cap Y = U_j\).
Let $w_j \colon V_j\to \mathbb{C}$ be local defining functions of $Y$, 
and 
$z_j=(z_j^1, z_j^2, \dots, z_j^d)$ be coordinates of $U_j$. 
Take a local frame $e_j$ of $L$ on each $V_j$. 
As $L|_Y$ admits a structure as unitary flat line bundle, one can take $e_j$'s such that $t_{jk}^{-1}\cdot e_k|_{U_{jk}}=e_j|_{U_{jk}}$ holds for some $t_{jk}\in\mathrm{U}(1)$ on each $U_{jk}:=U_j\cap U_k$. 

Then the transition functions admit an expansion
\[
t_{jk}^{-1} \frac{e_k}{e_j} = 1 + \sum_{\alpha \geq 1} f_{kj,\alpha}(z_j) \, w_j^\alpha.
\]
If the system \(\{(V_j, e_j, w_j)\}\) is of type \(n\), namely \(f_{kj,\alpha} \equiv 0\) for all \(\alpha < n\), the class defined in \cite{Koi21} is
\[
\left[\left\{(U_{jk}, f_{kj,n}(z_j)\cdot (dw_j)^{\otimes n})\right\}\right]
\in \check{H}^1(\{U_j\}, \mathcal{O}_Y(N^{-n}_{Y/X}))
\]
denoted by $u_n(Y, X, L; \{(V_j, e_j, w_j)\})$.

Let \(h\) be a smooth Hermitian metric on \(L\), and let \(\varphi_j\) be the local weight functions on \(V_j\). 
As $|t_{jk}|=1$, one has that 
\[
\varphi_k(z_k, w_k)
 = -\log |e_k|_{h(z_k, w_k)}^2
 = -\log |e_j|_{h(z_j, w_j)}^2-\log \left|1+\sum_{|\alpha|\geq 1} f_{kj, \alpha}(z_j)\cdot w_j^\alpha\right|^2
\]
if $(z_j, w_j)=(z_k, w_k)$. 
By considering Taylor expansion of the function $x\mapsto \log (1+x)$, one has that 
\begin{equation}\label{eq:diff_phi}
\varphi_k-\varphi_j = 
- f_{kj,n}(z_j)\cdot w_j^n
- \overline{f_{kj, n}(z_j)}\cdot \overline{w_j^n}
+O(|w_j|^{n+1})
\end{equation}
holds if $\{(V_j, e_j, w_j)\}$ is a system of type $n$.

Taking \(\partial\) of (\ref{eq:diff_phi}), we obtain
\begin{equation}\label{eq:diff_phi2}
\partial\varphi_k - \partial\varphi_j = -n f_{kj,n}(z_j)\, w_j^{n-1} dw_j + O(|w_j|^n).
\end{equation}
We claim that the image of $u_n(Y, X, L; \{(V_j, e_j, w_j)\})$ in $H^1\big(Y, i^{-1} \Omega^1_X \otimes_{i^{-1} \mathcal{O}_X} \mathcal{O}_X / \mathcal{I}_Y^n \big)$
is $u_n(Y,X,L)$ for $n \geq 2$.

First, 
the derivation 
$
d : \mathcal{O}_X \to \Omega_X^1
$
induces a morphism
\[
\mathcal{I}_Y \to i^* \Omega_X^1.
\]
Since \(d(\mathcal{I}_Y^2)|_Y = 0\), this factors through
\[
d : \mathcal{I}_Y / \mathcal{I}_Y^2 \to i^* \Omega_X^1.
\]
This map is injective and identifies \(\mathcal{I}_Y / \mathcal{I}_Y^2\) with the conormal bundle \(N^{-1}_{Y/X}\).
Locally, if \(w_j\) is a defining function of \(Y\), i.e. \(\mathcal{I}_Y = (w_j)\), then
\[
\mathcal{I}_Y / \mathcal{I}_Y^2 = \mathcal{O}_Y \cdot [w_j],
\quad
[w_j] \mapsto dw_j.
\]

For any \(n \geq 2\), there is a canonical identification
$
(N_{Y/X})^{-(n-1)} \cong \mathcal{I}_Y^{n-1} / \mathcal{I}_Y^{n},
$
under which the tensor power corresponds to multiplication in the ideal. Locally, this is given by
\[
(dw_j)^{\otimes (n-1)} \longmapsto [w_j^{n-1}].
\]
Via this identification,
the image of $u_n(Y, X, L; \{(V_j, e_j, w_j)\})$ in 
$$H^1(Y, i^* \Omega^1_X \otimes N^{-(n-1)}_{Y/X})\simeq H^1(Y, i^* \Omega^1_X \otimes \cI_Y^{n-1}/\cI_Y^{n}) \simeq H^1(Y, i^{-1} \Omega^1_X \otimes_{i^{-1} \cO_X} \cI_Y^{n-1}/\cI_Y^{n})$$
is 
\[
\left[\left\{\left(U_{jk},\ f_{kj, n}(z_j) w_j^{n-1}\cdot (dw_j) \right)\right\}\right]
\in H^1(Y, i^{-1} \Omega^1_X \otimes_{i^{-1} \cO_X} \cI_Y^{n-1}/\cI_Y^{n}).
\]
Via the natural inclusion \(\cI_Y^{n-1}/\cI_Y^n \hookrightarrow \mathcal{O}_X/\cI_Y^n\), this defines a class in
\[
H^1\bigl(Y, i^{-1}\Omega^1_X \otimes_{i^{-1}\mathcal{O}_X} \mathcal{O}_X/\cI_Y^n\bigr),
\]
given by the same {\v C}ech cocycle.
%In particular, its image in $H^1\big(Y, i^{-1} \Omega^1_X \otimes_{i^{-1} \mathcal{O}_X} \mathcal{O}_X / \mathcal{I}_Y^n \big)$
%is given by
%the same {\v C}ech cocycle.

Thus by Example \ref{ex1} and (\ref{eq:diff_phi2}), the above {\v C}ech cocycle coincides, up to a universal constant, with the {\v C}ech differential of \(\{\partial\varphi_j\}\), viewed as a \(0\)-cochain with values in
$$i^{-1} \Omega^1_X \otimes_{i^{-1} \mathcal{O}_X} \ec_{Y^{(n)}}^{0,0}$$ 
%is (up to a universal constant) the  {\v C}ech differential of
%$$\left\{\left(U_{j},\partial \varphi_j \right) \right\}$$
Passing to the Dolbeault resolution in Theorem \ref{AL} and applying \(\bar{\partial}\), this corresponds to the image of the Chern curvature under the morphism
$$i^{-1} \ec_{X}^{1,1} \to i^{-1} \Omega^1_X \otimes_{i^{-1} \mathcal{O}_X} \ec_{Y^{(n)}}^{0,1}$$
which represents \(u_n(Y,X,L)\). This shows the compatibility of the two constructions.
\end{remark}

%We will need the following notions of positivity on singular spaces (see \cite[Definition 4.6.2]{BEG13}).
%\begin{definition}
%Let $Y$ be a normal complex analytic space.  
%The Bott--Chern cohomology of $Y$ is defined by
%\[
%H^{1,1}_{\mathrm{BC}}(Y)
%:= H^{2}(Y, \R \cO_Y).
%\]
%Equivalently, one may use currents:
%\[
%H^{1,1}_{\mathrm{BC}}(Y)
%:= \frac{\{ T \in \cD'^{(1,1)}(Y) \mid d T = 0 \}}
%{\{ \sqrt{-1}\,\partial\bar{\partial} S \mid S \in \cD'^{(0,0)}(Y) \}}.
%\]
%Let $L$ be a holomorphic line bundle on $Y$.
%A singular Hermitian metric $h$ on $L$ is given by the following data: for any open set $U \subset Y$ and any local holomorphic frame $e$ of $L$ on $U$, there exists a function $\varphi \in L^1_{\mathrm{loc}}(U \cap Y_{\mathrm{reg}})$ such that
%$
%|e|_h^2 = e^{-\varphi},
%$
%where $\varphi$ is allowed to take the value $-\infty$. 
% One says that $L$ is pseudo-effective if there exists a (possibly singular) Hermitian metric $h$ on $L$ such that its curvature current
%\[
%\Theta_h(L) = -\partial\bar{\partial} \log h
%\]
%is a closed semi-positive $(1,1)$-current on $Y$. 
%\end{definition}
Note that the class $u_n(Y, X, L; \{(V_j, e_j, w_j)\})$ may depend on the choice of $\{(V_j, e_j, w_j)\}$ (see \cite[Lemma 3.1]{Koi21}).

One may propose the following variant definition of the Ueda type.

\begin{definition}
Let $Y$ be a reduced analytic subvariety of a complex manifold $X$, and let $\cI_Y$ be its ideal sheaf. For $i \geq 1$, consider the $i$-th infinitesimal neighborhood of $Y$ in $X$, defined by
\[
(Y^{(i)}, \cO_{Y^{(i)}}) := (Y, \cO_X / \cI_Y^i).
\]
Let $L$ be a holomorphic line bundle on $X$.
We define the generalized Ueda type of $L$ along $Y$ by
\[
\inf \{ k \geq 1 \, ; \, u_k(Y,X,L) \neq 0 \}.
\]
\end{definition}
Note that if $u_k(Y,X,L) \neq 0$, then $u_l(Y,X,L) \neq 0$ for any $l \geq k$.
\begin{remark}
By Remark \ref{coincide}, for the holomorphic line bundle $\cO(Y)$ on a complex manifold $X$ containing a smooth hypersurface $Y$ such that $\cO(Y)|_Y$ is unitary flat, the generalized Ueda type is greater than or equal to the Ueda type. 
For $n \geq 2$, although there exists a natural morphism
\[
H^1(Y, N^{-n}_{Y/X}) \to H^1\big(Y, i^{-1} \Omega^1_X \otimes_{i^{-1} \mathcal{O}_X} \mathcal{O}_X / \mathcal{I}_Y^n \big),
\]
this morphism is given as the composition of two natural maps, each of which arises from a long exact sequence, while the composition itself does not. For this reason, the converse direction remains unclear.
\end{remark}
We have the following variant of \cite[Theorem 3.7]{Koi21} using Remark \ref{curv-rep}.
Compared with \cite[Theorem 3.7]{Koi21}, the assumption that $Y$ is K\"ahler can be avoided.
The case for general line bundle $L$ with $c_1(N^{-1}_{Y/X})=0 \in H^{1,1}_{\mathrm{BC}}(Y)$
is also not discussed in \cite[Theorem 3.7]{Koi21}.
When $N^{-1}_{Y/X}$ is not pseudo-effective, the conclusion follows directly from \cite[Theorem 3.7]{Koi21} by Remark \ref{coincide}.
 \begin{corollary}
\label{2nd-ueda}
 Let \(X\) be a complex manifold, \(Y \subset X\) a compact smooth hypersurface, and \(L|_Y\) a unitary flat line bundle on \(Y\).
Assume that $N^{-1}_{Y/X}$ is not pseudo-effective or $c_1(N^{-1}_{Y/X})=0 \in H^{1,1}_{\mathrm{BC}}(Y)$.
 If \(L\) is  semi-positive,
$$u_1(Y,X ,L)=u_2(Y,X ,L)=0.$$ 
  \end{corollary}
\begin{proof}
By Corollary \ref{koi}, we have
\[
u_1(Y,X,L)=0.
\]
Thus, we shall focus on the study of \(u_2(Y,X,L)\).
We use the same notation as in the proof of Corollary \ref{koi}.

Consider a point $z \in Y$. Choose local holomorphic coordinates $(h,v)$ on $X$ such that $Y$ is locally defined by $\{v=0\}$. In these coordinates, the Chern curvature $\Theta(L,h)(z)$ can be written in block form as
\[
\begin{pmatrix}
0 & A \\
\overline{A}^t & *
\end{pmatrix},
\]
where $A(h,0)=0$.

The key observation of \cite[Theorem 3.7]{Koi21} is that the lower-right block also vanishes under the assumption that $N^{-1}_{Y/X}$ is not pseudo-effective. In particular, the derivative of this term with respect to the horizontal coordinates $h$ coincides, up to sign, with the derivative of $A$ with respect to the vertical coordinate $v$. Since the lower-right block vanishes, this derivative vanishes, which represents the generalized second Ueda obstruction class.
(When $N^{-1}_{Y/X}$ is unitary flat, it seems to be unclear that the lower-right block vanishes. For our purpose, this is enough to conclude that the lower-right block is pluriharmonic and conclude with similar calculation in \cite[Proposition 3.6]{Koi21}. We refer the end of the proof for more details.
However, if the lower-right block is nonzero, the Chern curvature representative has a nonzero zeroth-order term in its Taylor expansion in the normal $v-$direction, which prevents a direct generalization to higher-degree Ueda obstruction classes by induction on the order of the Taylor expansion.
In \cite[Section 3.5]{Koi21}, this is proved by a suitable choice of local trivializations and local coordinates in the case where the line bundle $L=\cO(Y)$.
)

Let us present the argument in detail within our notation and assumptions.

Let $\varphi$ be a local potential of $\Theta(L,h)$ near $z$.
The lower-right block corresponds to $\varphi_{v \overline{v}}|_Y$.

The Taylor expansion of $\varphi$ up to order $2$ in the $v$-direction is given by
\[
\varphi(h,v)
= \varphi^{(0,0)}
+ \varphi^{(1,0)} v
+ \varphi^{(0,1)} \overline{v}
+ \varphi^{(2,0)} v^2
+ \varphi^{(1,1)} |v|^2
+ \varphi^{(0,2)} \overline{v}^2
+ O(|v|^3)
\]
where $\varphi^{(\bullet,\bullet)}$ is independent of $v$.
Note that
\[
(\varphi)_{v \overline{v}}
= \varphi^{(1,1)}
+ O(|v|).
\]
Since $u_1(Y,X,L)=0$ is represented by the Chern curvature as in the proof of Corollary \ref{koi}, we have for any $a,b$,
\[
(\varphi^{(0,0)})_{h^a \overline{h}^b}=0, \quad
(\varphi^{(1,0)})_{\overline{h}^b}=0, \quad
(\varphi^{(0,1)})_{h^a}=0.
\]
It follows that
\[
(\varphi)_{h^a \overline{h}^b}
= (\varphi^{(2,0)})_{h^a \overline{h}^b}\cdot v^2
+ (\varphi^{(1,1)})_{h^a \overline{h}^b}\cdot |v|^2
+ (\varphi^{(0,2)})_{h^a \overline{h}^b}\cdot \overline{v}^2
+ O(|v|^3),
\]
\[
(\varphi)_{h^a \overline{v}}
= (\varphi^{(1,1)})_{h^a }\cdot v
+ 2(\varphi^{(0,2)})_{h^a }\cdot \overline{v}
+ O(|v|^2),
\]
\[
(\varphi)_{v \overline{h}^b}
= 2(\varphi^{(2,0)})_{ \overline{h}^b}\cdot v
+ (\varphi^{(1,1)})_{ \overline{h}^b}\cdot \overline{v}
+ O(|v|^2).
\]
Since $L$ is semi-positive, the matrix $\bigl( (\varphi)_{h^a \overline{h}^b} \bigr)_{a}^b$ is semi-positive definite.
By \cite[Lemma 2.1]{Koi21} (see also the proof of \cite[Lemma 3.5]{Koi21}), it follows that $\varphi^{(1,1)}$ is psh.

We emphasize that, unlike in \cite[Lemma 3.5]{Koi21}, we do not choose a special trivialisation of the line bundle $L$ near $z$.

A direct calculation as in \cite[Proposition 3.6]{Koi21} shows that, if $\varphi^{(1,1)} \not\equiv 0$, then $\log \varphi^{(1,1)}$ is psh.
We emphasize that the plurisubharmonicity relies essentially on the fact that $Y$ is a hypersurface of $X$.

Note that, since $Y$ is a smooth hypersurface of $X$, the regular locus of $Y^{(2)}$ coincides with $Y$, as the Cohen--Macaulay condition is automatically satisfied.
We claim that if $\varphi^{(1,1)} \not\equiv 0$ near some point of $Y$, then $\varphi^{(1,1)}$ glues to define a singular Hermitian metric on $N_{Y/X}$, which is semi-negative by the plurisubharmonicity of $\log \varphi^{(1,1)}$.
More precisely, in our situation, instead of \cite[(3.1)]{Koi21} which depends on a choice of local trivialization, the fact that the pointwise $(1,1)$-form $\varphi^{(1,1)}\, dv \wedge d \overline v$ glues to a singular metric on $N_{Y/X}$ follows from the observation that, for each fixed point $x$, the form $\varphi^{(1,1)}\, dv \wedge d \overline v$ is obtained as the image of the Chern curvature under the natural morphism
\[
\ec^{1,1}_{X,x} \longrightarrow \ec^{1,1}_{X,x} \otimes_{\mathcal{C}^{\infty}_{X,x}} \C,
\]
that is, by passing to the fiber at $x$, where $\C$ is regarded as a $\mathcal{C}^{\infty}_{X,x}$-module via the evaluation map $\mathrm{ev}_x$.
Note that the fact that the Chern curvature takes values in $N_{Y/X}^{-1} \otimes_{\mathcal{C}^{\infty}_{Y}} \overline{N_{Y/X}^{-1}}$ follows from the vanishing of the Chern curvature representative of the generalized Ueda first obstruction class.
This finishes the proof of our claim.

We complete the proof by distinguishing cases according to the positivity properties of the conormal bundle.

If $N_{Y/X}^{-1}$ is not pseudo-effective, the above claim leads to a contradiction, implying that $\varphi^{(1,1)} \equiv 0$.

By Remark \ref{curv-rep}, the generalized second Ueda obstruction class \(u_2(Y,X,L)\)
is represented by the image of the Chern curvature in \(i^{-1}\Omega_X^1 \otimes \ec^{0,1}_{Y^{(2)}}\); in local coordinates, this corresponds to the component of $(\varphi)_{v \overline{h}^b}$ of order one in the normal direction.
Note that by \cite[Lemma 2.1]{Koi21}, the component of $(\varphi)_{h^a \overline{h}^b}$ of order one in the normal direction vanishes.

Thus the generalized second Ueda obstruction class \(u_2(Y,X,L)\)
 is locally represented near $z$ by $\sum_b 2 (\varphi^{(2,0)})_{\overline{h}^b}v d\overline{h}^b $. This vanishes, as follows from the proof of \cite[Proposition 3.6, line 5, p.~2257]{Koi21}.
Since this represents the generalized second Ueda obstruction class,
$$u_2(Y,X ,L)=0.$$ 
This completes the proof in the case where $N_{Y/X}^{-1}$ is not pseudo-effective.

Assume that $c_1(N^{-1}_{Y/X})=0 \in H^{1,1}_{\mathrm{BC}}(Y)$ and that $\varphi^{(1,1)} \not\equiv 0$ near some point. Then
\[
\partial\bar{\partial} \log \varphi^{(1,1)} \equiv 0,
\]
since it is the unique semi-positive $(1,1)$-current in the class $c_1(N^{-1}_{Y/X})$.
By the last inequality in the proof of \cite[Proposition 3.6, line 5, p.~2257]{Koi21}, we obtain
\[
\sum_b 2 (\varphi^{(2,0)})_{\overline{h}^b} \, d\overline{h}^b \equiv 0,
\]
which concludes the proof in this case.
\end{proof}
\begin{remark}
Note that the same arguments as in Corollary \ref{2nd-ueda} also apply at a smooth point of $Y$ even when $Y$ is not smooth. However, as noted in Remark \ref{non-reduced-diff}, it seems difficult to conclude that the generalized second Ueda class is trivial.
\end{remark} 
\begin{remark}
Assume that $N^{-1}_{Y/X}$ is not pseudo-effective. Then one can slightly strengthen the conclusion to
\[
u_3(Y,X,L)=0.
\]
We give the proof using the same notation as in Corollary \ref{2nd-ueda}.

Let $\varphi$ be a local potential of $\Theta(L,h)$ near $z$.
The Taylor expansion of $\varphi$ up to order $3$ in the $v$-direction is
\[
\varphi(h,v)
=\sum_{p,q \in \mathbb{N},\, p+q \leq 3} \varphi^{(p,q)} v^p \overline{v}^q
+ O(|v|^4),
\]
where each $\varphi^{(p,q)}$ is independent of $v$.

By the proof of Corollary \ref{2nd-ueda}, we have for any $a,b$,
\[
(\varphi)_{v \overline{v}}
=2 \varphi^{(2,1)} v+2\varphi^{(1,2)}\overline{v}
+ O(|v|^2),
\]
\[
(\varphi)_{h^a \overline{h}^b}
= (\varphi^{(3,0)})_{h^a \overline{h}^b}\, v^3
+ (\varphi^{(2,1)})_{h^a \overline{h}^b}\, v^2 \overline{v}
+ (\varphi^{(1,2)})_{h^a \overline{h}^b}\, v \overline{v}^2
+ (\varphi^{(0,3)})_{h^a \overline{h}^b}\, \overline{v}^3
+ O(|v|^4),
\]
\[
(\varphi)_{h^a \overline{v}}
= (\varphi^{(2,1)})_{h^a }\, v^2
+2(\varphi^{(1,2)})_{h^a }\, v \overline{v}
+ 3(\varphi^{(0,3)})_{h^a }\, \overline{v}^2
+ O(|v|^3),
\]
\[
(\varphi)_{v \overline{h}^b}
= 3(\varphi^{(3,0)})_{ \overline{h}^b}\, v^2
+2 (\varphi^{(2,1)})_{ \overline{h}^b}\, v\overline{v}
+ (\varphi^{(1,2)})_{ \overline{h}^b}\, \overline{v}^2
+ O(|v|^3).
\]

By \cite[Lemma 2.1]{Koi21}, the condition $(\varphi)_{v \overline{v}} \geq 0$ implies
\[
\varphi^{(2,1)} = \varphi^{(1,2)} = 0.
\]

Fix $a$. The restriction of the complex Hessian of $\varphi$ to the complex plane spanned by $\frac{\partial}{\partial v}$ and $\frac{\partial}{\partial h^a}$ is given by
\[
\begin{pmatrix}
(\varphi)_{v \overline{v}} & (\varphi)_{v \overline{h^a}}\\
(\varphi)_{h^a \overline{v}} & (\varphi)_{h^a \overline{h^a}}
\end{pmatrix},
\]
which is semi-positive by assumption.

The Taylor expansion of the determinant of this matrix starts with the term
\[
-9\,\big|(\varphi^{(3,0)})_{\overline{h}^a}\big|^2\, |v|^4,
\]
which implies
\[
(\varphi^{(3,0)})_{\overline{h}^a}=(\varphi^{(0,3)})_{h^a}=0.
\]

On the other hand, the generalized third Ueda obstruction class $u_3(Y,X,L)$ is locally represented near $z$ by
\[
\sum_b 3 (\varphi^{(3,0)})_{\overline{h}^b} \, v^2 d\overline{h}^b.
\]
Since $(\varphi^{(3,0)})_{\overline{h}^b}=0$, we conclude that
\[
u_3(Y,X,L)=0.
\]
\end{remark}
To conclude this section, for the convenience of the reader, we recall the following example from \cite[p.\,2262]{Koi21}, which shows that $\varphi^{(1,1)}$ need not vanish.
(This result should be compared with \cite[Theorem 1.2]{Koi22}, where it is shown that the existence of a semi-positive metric on \(\cO(Y)\) forces the Ueda type to be infinite, although \(\varphi^{(1,1)}\) need not vanish.)
\begin{example}
\label{ex2}
Let $X$ be a surface and let $Y \subset X$ be a non-singular compact curve holomorphically embedded in $X$ with topologically trivial normal bundle, such that the pair $(Y,X)$ is of type $(\beta')$ or $(\beta'')$ in the classification of \cite[Section 5]{Ued83}. That is, there exists an open covering $\{U_j\}$ of $Y$ and, for each $j$, a local defining function $w_j$ of $Y$ on a neighborhood $V_j$ of $U_j$ such that
\[
t_{jk} w_k = w_j \quad \text{on } V_j \cap V_k,
\]
where $t_{jk} \in \mathrm{U}(1)$.

Consider a $C^\infty$ Hermitian metric $h$ on $\cO(Y)$ whose local weight functions $\varphi_j$ on each $V_j$, with respect to the local frame $e_j$ corresponding to the meromorphic function $1/w_j$, satisfy
\[
\varphi_j = |w_j|^2.
\]
Then $\varphi^{(1,1)} = 1$, and in particular it does not vanish.
Moreover, the Taylor expansion of the corresponding Chern curvature shows that $u_k(Y,X,\cO(Y))=0$ for all $k \geq 1$.
\end{example}
  
Department of Mathematics, Graduate School of Science, Osaka Metropolitan University, 3-3-138 Sugimoto, Osaka 558-8585, Japan.\\
Email address: xiaojun.wu@univ-cotedazur.fr; y25161q@omu.ac.jp.  
  \end{document}